\documentclass[12pt]{article}
\usepackage{mathtools}
\usepackage[francais,english]{babel}
\usepackage[T1]{fontenc}
\usepackage{kpfonts}
\usepackage{color}
\usepackage{amssymb}
\usepackage{amsthm}
\usepackage{geometry}
\usepackage[utf8]{inputenc}
\RequirePackage{filecontents}
\usepackage{amsmath}
\usepackage{graphicx}
\usepackage{sectsty}
\usepackage{lipsum}
\usepackage{mathrsfs}
\usepackage{amsthm}
\allsectionsfont{\centering}
\newtheorem{proposition}{Proposition}[section]
\newtheorem{remark}{Remark}[section]
\newtheorem{lemma}{Lemma}[section]

\numberwithin{equation}{section}
\begin{document}
\large
\begin{center}
   \textbf{\Large{Spectral density for the Schr\"{o}dinger operator with magnetic field in the unit complex ball: Solutions of evolutionary equations and applications to special functions }}
\end{center}
\begin{center}
\hrule
\end{center}
\begin{center}
   \hspace{-0.25cm}\textbf{Nour eddine Askour $^{1,2,a}$, Mohamed Bouaouid $^{1,b}$ and  Abdelkarim Elhadouni $^{1,c}$}
\end{center}
\begin{center}
$^{1}$ Department of Mathematics, Sultan Moulay Slimane University, Faculty
		of Sciences and Technics, Beni Mellal, BP 523, 23000, Morocco.
\end{center}
\begin{center}
$\&$
\end{center}
\begin{center}
$^{2}$ Department of Mathematics, Mohammed V University, Faculty of Sciences, Rabat, P.O. Box 1014, Morocco.
\end{center}
\begin{center}
\hspace{-1cm}$^{a}$ n.askour@usms.ma,\hspace{0.25cm}$^{b}$ bouaouidfst@gmail.com and \hspace{0.25cm}$^{c}$ Karimfstgma@gmail.com
\end{center}
\section*{Abstract}
The Schwartz kernel of the spectral density for the Schr\"{o}dinger operator with magnetic field in the $n-$dimensional complex ball is given. As applications, we compute the heat, resolvent and the wave kernels. Moreover, the resolvent and wave kernels are used to establish two new formulas for the Gauss-hypergeometric function.
\begin{description}
  \item[Keywords:] Spectral density; Self-adjoint operator; Fourier-Helgason transform; G-invariant Laplacian; Heat kernel; Wave kernel; Gauss-hypergeometric function; Reproducing kernels. 	
  \item[2010 Mathematics Subject Classification:] 32A70; 58C40; 49K20; 34K08.
\end{description}
\newpage
\tableofcontents
\newpage
\section{Introduction}
The spectral density associated with a self-adjoint operator $T$ is considered as a powerful tool for developing more general functional calculus \cite{Est}. In particular, it gives the possibility to define the operator function $f(T)$ for a more generalized function $f$. It is well known that in quantum theory that each observable is represented by a self-adjoint operator on a complex Hilbert space. Thus for the study of quantum dynamic, the spectral density becomes a very efficient tool to solve evolutionary problems associated with such observable.\\
Recall that the Von Neumann spectral theorem says that any self-adjoint operator $T$ densely defined on a complex Hilbert space $H$ has a Stieltjes integral representation
\begin{align}\label{E1.1}
T=\int_{-\infty}^{+\infty } \lambda dE_{\lambda},\hspace{0.2cm}I=\int_{-\infty}^{+\infty}dE_{\lambda},
\end{align}
with respect to the some unique spectral measure $\left\lbrace E_{\lambda}\right\rbrace_{\lambda}$ (\cite{Lei,Kon,Mor}). The spectral density
\begin{align}\label{E1.2}
e_{\lambda}=\dfrac{dE_{\lambda}}{d\lambda},
\end{align}
is understood as an operator-valued distribution. It is an element of the space $\mathrm{\mathcal{D}'}\left(\mathbb{R},L\left(D(T),H\right)\right)$, where $L(D(T),H)$ is the space of bounded operator from $D(T)$ the domain of selfadjointness for T, to the whole Hilbert space $H$. In the case where  $T$ is constant-coefficients elliptic operator with symbol $P$ on the whole space $\mathbb{R}^{n}$, i.e $T=P(i\partial)$, where
$i\partial=i(\frac{\partial}{\partial x_{1}},...,\frac{\partial}{\partial x_{n}})$. Then, the Schwartz kernel of the spectral density associated with its self-adjoint extension can be written as \cite{Est}
\begin{align}\label{E1.3}
e(\lambda,x,y)=\frac{1}{(2\pi)^{n}}\left<e^{i(x-y).\xi},\delta(P(\xi)-\lambda)\right>,
\end{align}
where $\delta(P(\xi)-\lambda)$ is the Layer distribution associated with the function $P(\xi)-\lambda$ and the dot product $(x-y).\xi$ means the classical scalar product in $\mathbb{R}^{n}$.\\
For the general case, the explicit computation of the spectral density $e_{\lambda}$ in the concrete cases depends on the spectral tools which we dispose of. For example, based on the explicit resolvent kernel of the Khon-Laplacian in \cite{Ask} the authors have used the Cauchy-integral representation of the resolvent operator with respect to the spectral measure for computing the spectral density. Also the method used in \cite{Ask} has been fully adapted to the case of Heisenberg quaternionic Laplacian \cite{Zak}. In \cite{Ask2} the author has used the explicit expression of the spectral family associated with the Schr\"{o}dinger operator with the magnetic field in the n-Euclidean complex space, to get by a direct computation the corresponding spectral density.\\
In the present work, the explicit expression of the inverse Fourier-Helgason transform helps us to compute the spectral density for a Schr\"{o}dinger operator with magnetic field on the complex hyperbolic space. Precisely, we will be concerned  with the following class of  Schr\"{o}dinger operator
\begin{align}\label{E1.4}
\Delta_{\nu}=4(1-\mid z\mid^{2})\{\sum_{1\leq i,j\leq n}(\delta_{ij}-z_{i}\bar{z}_{j})\frac{\partial^{2}}{\partial z_{i}\partial \bar{z}_{j}}-\nu\sum_{j=1}^{n}\bar{z}_{j}\frac{\partial}{\partial \bar{z}_{j} }\},
\end{align}
in the unit ball $\mathbb{B}_{n}$ of the $n-$dimensional complex space, where the parameter $\nu$ is such that  $\nu>n$ and $\nu\in\mathbb{R}\backslash\mathbb{Z}$. The operator $\Delta_{\nu}$ will be viewed as a linear operator  acting on the Hilbert space $L^{2}_{\nu}(\mathbb{B}_{n}):=L^{2}(\mathbb{B}_{n},d\mu_{\nu}(z))$, where $d\mu_{\nu}=(1-\mid z \mid^{2})^{\nu-n-1}dm(z)$ with $dm(z)$ is the Lebesgue measure on $\mathbb{C}^n=\mathbb{R}^{2n}$. This operator has been largely considered in several equivalent forms by many authors in different contexts \cite{Zha,Ahe,Gha,Int}.\\
In order to make the operator $\Delta_{\nu}$ self-adjoint on the Hilbert space $L^{2}_{\nu}(\mathbb{B}_{n})$, we must consider it on its maximal domain
\begin{align}\label{E1.5}
\mathcal{D}=\lbrace  F\in L_{\nu}^{2}(\mathbb{B}_{n}), \hspace{0.5cm}\Delta_{\nu}F\in L_{\nu}^{2}(\mathbb{B}_{n}) \rbrace.
\end{align}
Then, the spectrum of the operator $\Delta_{\nu}$ decomposes as
\begin{align}\label{E1.6}
\sigma(\Delta_{\nu})=\sigma_{p}(\Delta_{\nu})\cup \sigma_{c}(\Delta_{\nu}),
\end{align}
where the point spectrum $\sigma_{p}(\Delta_{\nu})$ is given by
\begin{align}\label{E1.7}
\sigma_{p}(\Delta_{\nu})=\{-(\lambda_{\ell}^{2}+(n-\nu)^{2}),\hspace{0.25cm}\lambda_{\ell}=-i(\nu-n-2\ell),\hspace{0.25cm}
\ell=0,...,[\frac{\nu-n}{2}]\},
\end{align}
and the continuous spectrum is given by
\begin{align}\label{E1.8}
\sigma_{c}(\Delta_{\nu})=\{-(\lambda^{2}+(\nu-n)^{2}),\hspace{0.25cm}\lambda\in \mathbb{R}\}.
\end{align}
In view point of spectral theory, the study of the operator $\Delta_{\nu}$ on the domain $\mathcal{D}$ is equivalent to the study of the following operator on the same domain
\begin{align}\label{E1.9}
\tilde{\Delta}_{\nu}=-(\Delta_{\nu}+(n-\nu)^{2}),
\end{align}
which admits the set of positive real numbers as continuous spectrum. Thus, it suffice to compute the spectral density for the operator $\tilde{\Delta}_{\nu}$.\\
Before going a head, we give a concise picture of our results. Let then consider the spectral decomposition
\begin{align}\label{E1.10}
\tilde{\Delta}_{\nu}=\int_{-\infty}^{+\infty}sdE_{s}^{\nu},\hspace{0.2cm}I=\int_{-\infty}^{+\infty}dE_{s}^{\nu},
\end{align}
where $\{E_{s}^{\nu},\hspace{0.2cm}s\in\mathbb{R}\}$ is the unique corresponding spectral measure. We are precisely concerned with the Schwartz kernel of the spectral density
 \begin{align}\label{E1.11}
 e_{s}^\nu=\dfrac{dE_{s}^{\nu}}{ds},
 \end{align}
which is an operator valued distribution, element of $\mathcal{D'}(\mathbb{R},L(\mathcal{D}),\mathrm{L}_{\nu}^{2}(\mathbb{B}_{n}))$. We prove that this Schwartz kernel can be expressed as
\begin{align}\label{E1.12}
\nonumber e^\nu(s,w,z)&=\frac{\Gamma(n)}{4\pi^{n+1}2^{2(\nu-n)}}(1-<z,w>)^{-\nu}|C_{\nu}(\sqrt{s})|^{-2}(\sqrt{s})^{-1}
\chi_{+}(s)\phi_{\sqrt{s}}^{(n-1,-\nu)}(d(z,w))\\
&+\sum_{j=0}^{[\frac{\nu-n}{2}]}c_{j}\frac{j!}{(n)_{j}}(1-<z,w>)^{-\nu}P_{j}^{(n-1,-\nu)}(\cosh(2d(z,w)))\delta(s-s_{j}),
\end{align}
with
\begin{align}\label{E1.13}
 s_{j}=-(2j+\nu-n)^{2};j=0,...,[\frac{\nu-n}{2}],
 \end{align}
\begin{align}\label{E1.14}
c_{j}=\frac{2\Gamma(n+j)}{\pi^{n}\Gamma(n)j!}\frac{(\nu-n-2j)\Gamma(\nu-j)}{\Gamma(\nu-n-j+1)},
\end{align}
$\chi_{+}(s)$ is the characteristic function of the set of positive real numbers and
\begin{align}\label{E1.15}
\phi^{(\alpha,\beta)}_{\lambda}(t)=\prescript{}{2}{F}_1^{}\left( \frac{\alpha+\beta+1-i\lambda}{2},\frac{\alpha+\beta+1+i\lambda}{2},1+\alpha;-\sinh^{2}t\right),
\end{align}
 is the Jacobi function (\cite[p.5]{Koo}).  $P^{(\alpha,\beta)}_j(x)$ denote the classical Jacobi polynomial of degree $j$ \cite{Bea} and $d(z,w)$ is the distance  corresponding to the Bergmann  metric in the unit ball $\mathbb{B}_{n}$ (\cite[p.25]{Zhu})
\begin{align}\label{E1.16}
\cosh^{2}(d(z,w))=\frac{\mid 1-<z,w>\mid^{2}}{(1-\mid z\mid^{2})(1-\mid w\mid^{2})}, (z,w)\in\mathbb{B}_{n}\times\mathbb{B}_{n},
\end{align} The function
\begin{align}\label{E1.17}
c_{\nu}(\lambda)=\dfrac{2^{-\nu+n-i\lambda}\Gamma(n)\Gamma(i\lambda)}{\Gamma(\frac{i\lambda+n-\nu}{2})
\Gamma(\frac{i\lambda+n+\nu}{2})}
\end{align}
is the analogous of the Harish-Chandra C-function defined in \cite{Zha}. For more detail, we can also see \cite{Koo}.
As a consequence of the spectral density, we can give for the operator $\Delta_{\nu}$ the heat semigroup
\begin{align}\label{E1.18}
e^{t\Delta_{\nu}}[\varphi](z)=\int_{\mathbb{B}_{n}}K_{\nu}(t,z,w)\varphi(w)d\mu_{\nu}(w),
\end{align}
where the heat kernel $K_{\nu}(t,z,w)$ is given by
\begin{align}\label{E1.19}
\nonumber K_{\nu}(t,z,w)&=(1-<z,w>)^{-\nu}\sum_{j=0}^{[\frac{\nu-n}{2}]}\tau_{j}e^{4j(j+n-\nu)t}P_{j}^{(n-1,-\nu)}(\cosh(2d(z,w))\\
\nonumber &+(1-<z,w>)^{-\nu}e^{-t(\nu-n)^{2}}\frac{\Gamma(n)}{2\pi^{n+1}2^{2(\nu-n)}}\\
&\times \int_{0}^{+\infty}e^{-t\lambda^{2}}\mid C_{\nu}(\lambda)\mid^{-2}\prescript{}{2}{F}_1^{}\left( \frac{n-\nu-i\lambda}{2},\frac{n-\nu+i\lambda}{2},n;-\sinh^{2}(d(z,w))\right)d\lambda,
\end{align}
with the constant $\tau_{j}=\frac{2(\nu-n-2j)\Gamma(\nu-j)}{\pi^{n}\Gamma(\nu-n-j+1)}$.
From the above heat kernel formula, we give also the resolvent kernel of the operator $\Delta_\nu$. Precisely, we have
 \begin{align}\label{E1.20}
(\xi-\Delta_{\nu})^{-1}[F](z)=\int_{\mathbb{B}_{n}}R_{\nu}(\xi,z,w)F(w)d\mu_{\nu}(w),\hspace{0.2cm}Re(\xi)> \omega(\nu,n),
\end{align}
where  $\omega(\nu,n)=p_{n}^{\nu}-(n-\nu)^{2}$ with $p_{n}^{\nu}=\max\{\mid s_{j}\mid,\hspace{0.2cm}0\leq j <\frac{\nu-n}{2}\}$ and $s_{j}=-(2j+n-\nu)^{2}$. The resolvent kernel $R(\xi,z,w)$ is given as follows
\begin{align}\label{E1.21}
\nonumber R(\xi,z,w)&=(1-\left<z,w\right>)^{-\nu}\sum_{j=0}^{[\frac{\nu-n}{2}]}\tau_j\frac{1}{\lambda_{j}^{2}+(\nu-n)^2+\xi}P_j^{(n-1,-\nu)}(\cosh(2d(z,w)))\\
\nonumber&+(1-\left<z,w\right>)^{-\nu}\frac{\Gamma(n)}{2\pi^{n+1}2^{2(\nu-n)}}\int_0^{+\infty}
\frac{|c_\nu(\lambda)|^{-2}}{\lambda^2+(\nu-n)^2+\xi}\\
& \times\prescript{}{2}{F}_1^{}\left(\frac{n-\nu-i\lambda}{2},\frac{n-\nu-i\lambda}{2},n,-\sinh^2(d(z,w))\right)d\lambda,
\end{align}
where  $\lambda_{j}=-i(\nu-n-2j)$ for $0\leq j<\frac{\nu-n}{2}$.\\
Also by using the spectral density, we solve for the operator $\Delta_{\nu}$ the following wave Cauchy problem
\begin{align}\label{E1.22}
\left\{
  \begin{array}{ll}
    \dfrac{\partial ^{2} u(t,z)}{\partial t^{2}}=\vartriangle _{\nu }u(t,z), \hspace{0.2cm}(t,z)\in \mathbb{R}\times\mathbb{B}_{n},&\hbox{} \\\\
   u(0,z)=0, \hspace{0.2cm}\dfrac{\partial u(0,t) }{\partial t}=f(z)\in C^{\infty}_{0}(\mathbb{B}_{n}). & \hbox{}
  \end{array}
\right.
\end{align}
Precisely, we have
\begin{align}\label{E1.23}
u(t,z)=\int_{\mathbb{B}_{n}}W_{\nu}(t,z,w)F(w)d\mu_{\nu}(w),
\end{align}
where the wave kernel is given by
\begin{align}\label{E1.24}
\nonumber W_{\nu}&(t,z,w)=(1-<z,w>)^{-\nu}\sum_{j=0}^{[\frac{\nu-n}{2}]}\tau_{j}
\dfrac{\sin(2t\sqrt{j(\nu-n-j)}}{2\sqrt{j(\nu-n-j)}}P_{j}^{(n-1,-\nu)}(\cosh(2d(z,w)))\\
\nonumber &+(1-<z,w>)^{-\nu}\frac{\Gamma(n)}{2\pi^{n+1}2^{2(\nu-n)}}\\
\nonumber \\
&\times \int_{0}^{+\infty}\dfrac{\sin(t\sqrt{\lambda^2+(n-\nu)^{2}})}{\sqrt{\lambda^2+(n-\nu)^{2}}}\mid C_{\nu}(\lambda)\mid^{-2}\prescript{}{2}{F}_1^{}\left( \frac{n-\nu-i\lambda}{2},\frac{n-\nu+i\lambda}{2},n;-\sinh^{2}(d(z,w))\right) d\lambda.
\end{align}
We end this summary  by given an application to special functions. Precisely, we obtain (up to our knowledge) two new integral formulas for the Gauss-hypergeometric function as will be stated below in the two following propositions.
\begin{proposition}\label{P1.1} Let $\nu>n$, $\nu\in \mathbb{R}\setminus\mathbb{Z}$ and $0\leq x<\sinh(|t|)$ for $t\in \mathbb{R}$, then we have the following integral formula
\begin{align}\label{E1.25}
\nonumber &\int_{0}^{+\infty}\left|\frac{\Gamma(\frac{i\lambda+n-\nu}{2})\Gamma(\frac{i\lambda+n+\nu}{2})}{\sqrt{\lambda}\Gamma(i\lambda)}\right|^{2}
\times\prescript{}{2}{F}_1^{}\left( \frac{n-\nu-i\lambda}{2},\frac{n-\nu+i\lambda}{2},n;-x\right)\sin(t\lambda)d\lambda\\
\nonumber &=(-1)^{n-1}\pi\frac{\Gamma(n-\frac{1}{2})}{\Gamma(n)}(1+x)^{\frac{\nu-n}{2}}(\frac{\sinh^{2}(t)}{1+x}-1)^{-n+\frac{1}{2}}_{+}\\
\nonumber &\times\prescript{}{2}{F}_1^{}\left(1-n+\nu,1-n-\nu,\frac{3}{2}-n,\frac{1}{2}-\frac{\cosh(t)}{2\sqrt{1+x}}\right)\\
&-2^{2(\nu-n+1)}\frac{\pi}{\Gamma(n)}\sum_{j=0}^{[\frac{\nu-n}{2}]}\frac{(\nu-n-2j)\Gamma(\nu-j)}{\Gamma(\nu-n-j+1)} P^{(n-1,-\nu)}_{j}(2x+1)\frac{\sinh(t(2j+n-\nu))}{2j+n-\nu}.
\end{align}
\end{proposition}
\begin{proposition}\label{P1.2} Let $\nu>n$, $\nu\in \mathbb{R}\setminus\mathbb{Z}$ and $\mu$ is a complex number such that $\mu\neq-i(2\ell+n\pm\nu)$ for $\ell=0,1,2,...$ and $Re(\mu^{2})<-P_{n}^{\nu}$, with $p_{n}^{\nu}=\max\{\mid s_{j}\mid,\hspace{0.2cm}0\leq j <\frac{\nu-n}{2}\}$ and $s_{j}=-(2j+n-\nu)^{2}$. Then, we have the following integral formula
\begin{align}\label{E1.26}
\nonumber&\int_0^{+\infty}\left|\frac{\Gamma(\frac{i\lambda+n-\nu}{2})
\Gamma(\frac{i\lambda+n+\nu}{2})}{\Gamma(i\lambda)}\right|^{2}\frac{1}{\lambda^2-\mu^2}\times
 \prescript{}{2}{F}_1^{}\left(\frac{n-\nu-i\lambda}{2},\frac{n-\nu+i\lambda}{2},n,-x\right)d\lambda\\
\nonumber&=\pi2^{2(\nu-n)}\frac{\Gamma(\frac{n-i\mu+\nu}{2})\Gamma(\frac{n-i\mu-\nu}{2})}{\Gamma(n)\Gamma(1-i\mu)}
(1+x)^{\frac{\nu+i\mu}{4}-\frac{n}{2}}\\
\nonumber&\times\prescript{}{2}{F}_1^{}\left(\frac{n-i\mu+\nu}{2},\frac{n-i\mu-\nu}{2},1-i\mu,\frac{1}{1+x}\right)\\
&-\frac{4\pi}{\Gamma(n)}2^{2(\nu-n)}\sum_{j=0}^{[\frac{\nu-n}{2}]}\frac{(\nu-n-2j)\Gamma(\nu-j)}{\Gamma(\nu-n-j+1)}\frac{1}
{\lambda_{j}^{2}-\mu^{2}}P_j^{(n-1,-\nu)}(2x+1).
\end{align}
\end{proposition}
This paper is summarized as follows. In section 2, we review some results for the Laplacian $\Delta_\nu$ in the frame work of $L^{2}$-concrete harmonic analysis. In section 3, we give the expression of the spectral density for the operator
$\tilde{\Delta}_{\nu}$. Section 4 is devoted to some applications as the heat semi-group, resolvent and wave kernels. In section 5, a special application of our results is reserved to establish two new integral formulas for the Gauss-hypergeometric function.
\section{ $L^{2}$-Concrete harmonic  analysis of the invariant Laplacian $\Delta_{\nu}$ and spectral interpretation}
In this section we review some results on the $L^{2}$-concrete harmonic analysis of the Laplacian $\Delta _{\nu}$ defined in (\ref{E2.12}). To do so, we endow the Euclidian complex space $\mathbb{C}^{n}$ with the inner product
\begin{align}\label{E2.1}
<z,w>=z_{1}\overline{w}_{1}+z_{2}\overline{w}_{2}+...+z_{n}\overline{w}_{n}.
\end{align}
The open unit ball $\mathbb{B}_{n}$ in $\mathbb{C}^{n}$ consists of $n-$tuples
$z=\left(
  \begin{array}{c}
    z_{1} \\
    \vdots \\
    z_{n} \\
  \end{array}
\right)$
for which
\begin{align}\label{E2.2}
\vert z\vert^{2}=\vert z_{1}\vert^{2}+...+\vert z_{n}\vert^{2}<1.
\end{align}
The ball $\mathbb{B}_{n}$ can be identified with the unit ball of $\mathbb{R}^{2n}$ and thus can be equipped with the Lebesgue measure $dm(z)=r^{2n-1}drd\sigma(w)$ where $d\sigma$ is the rotation invariant measure on the sphere $\partial\mathbb{B}_{n}=S^{2n-1}$.\\
The group $G:=SU(1,n)$ consists all matrices with determinant equal one and preserve the sesquilinear form
\begin{align}\label{E2.3}
<z,Jw>=z_{1}\overline{w}_{1}+z_{2}\overline{w}_{2}+...+z_{n}\overline{w}_{n}-z_{n+1}\overline{w}_{n+1},
\end{align}
where
\begin{align}\label{E2.4}
J=\left(
  \begin{array}{cc}
    I_{n} & 0 \\
    0 & -1\\
  \end{array}
\right).
\end{align}
We usually write a matrix $g$ in $SU(1,n)$ in block form as
\begin{align}\label{E2.5}
g=\left(
  \begin{array}{cc}
    A & B \\
    C & D\\
  \end{array}
\right).
\end{align}
Where $A, B, C$ and $D$ are $n\times n$, $n\times 1$, $1\times n$ and $1\times 1$ matrixes with complex entries, respectively. Then by definition $g$ is in $SU(1,n)$ if and only if $ det(g) =1$ and $g^{*}Jg=J$. This immediately gives that
\begin{align}\label{E2.6}
g^{-1}=Jg^{*}J=\left(
  \begin{array}{cc}
    A^{*} & -C^{*} \\
    -B^{*} & D^{*}\\
  \end{array}
\right).
\end{align}
The group $SU(1,n)$  acts transitively on $\mathbb{B}_{n}$ via the fractional linear transformations
\begin{align}\label{E2.7}
g.z =
\left(
  \begin{array}{cc}
    A & B \\
    C & D \\
  \end{array}
\right)z:=(Az+B)(Cz+D)^{-1}.
\end{align}
Recall that for
$g=
\left(
  \begin{array}{cc}
    A & B \\
    C & D \\
  \end{array}
\right)$
the above action satisfy the following relation
\begin{align}\label{E2.8}
1-<g.z,g.w>=\dfrac{1-<z,w>}{(Cz+D)\overline{(Cw+D)}}.
\end{align}
We denote by $U(1)$ and $U(n)$ the set of unimodular complex numbers and the set of complex unitary matrices, respectively. The origin in $\mathbb{C}^{n}$ is noted by $0$. Recall that
the stabilizer $K$ of $0\in \mathbb{B}_{n}$ is given  by
\begin{align}\label{E2.9}
K=\left\{\left(
      \begin{array}{cc}
        A & 0 \\
        0 & B \\
      \end{array}
    \right),\hspace{0.2cm}  A\in U(n),\hspace{0.2cm} B\in U(1)\hspace{0.2cm}\mbox{and}\hspace{0.2cm} det(AB)=1\right\}
\end{align}
which is the maximal compact subgroup of $SU(1,n)$. It is well known that the space $SU(1,n)/K$ is holomorphically isometric to $\mathbb{B}_{n}$ as Bargmann  ball \cite{Vol}. Also, the Bargmann ball $\mathbb{B}_{n}$ can be viewed
as a complex hyperbolic space. Thus, we have the following identification
\begin{align}\label{E2.10}
\mathbb{B}_{n}=SU(1,n)/K.
\end{align}
The group $G$ acts unitarly on the Hilbert space $L^{2}_{\nu}(\mathbb{B}_{n}):=L^{2}(\mathbb{B}_{n},d\mu_{\nu})$, where  $d\mu_{\nu}=(1-\mid z \mid^{2})^{\nu-n-1}dm(z)$ with $dm(z)$ is the Lebesgue measure on $\mathbb{C}^n$ via
\begin{align}\label{E2.11}
T^{\nu}(g)F(z)=J(g^{-1},z)^{\frac{\nu}{n+1}}F(g^{-1}.z),
\end{align}
where $\nu$ is a non integer real parameter such that $\nu>n$ and $J(g^{-1},z)$ is the complex Jacobian of the matrix $g^{-1}$. Then $T$ is a unitary representation of $G=SU(1,n)$ in \cite{Zha}. For more detail, we refer to \cite{Chr} and references therein. The corresponding $G-$invariant Laplacian was found in \cite{Zha}. It is of the form
\begin{align}\label{E2.12}
\Delta_{\nu}=4(1-\mid z\mid^{2})\{\sum_{1\leq i,j\leq n}(\delta_{ij}-z_{i}\bar{z}_{j})\frac{\partial^{2}}{\partial z_{i}\partial \bar{z}_{j}}-\nu\sum_{j=1}^{n}\bar{z}_{j}\frac{\partial}{\partial \bar{z}_{j} }\}.
\end{align}
 As mentioned in the introduction, we are precisely concerned with the G-invariant self adjoint operator
\begin{align}\label{E2.13}
 \tilde{\Delta}_{\nu}=-(\Delta_{\nu}+(\nu-n)^{2}),
 \end{align}
maximally defined on the domain
\begin{align}\label{E2.14}
\mathcal{D}=\{F\in L_{\nu}^{2}(\mathbb{B}_{n}),\hspace{0.2cm} \Delta_{\nu}F\in L_{\nu}^{2}(\mathbb{B}_{n})\}.
\end{align}
According to \cite{Zha} a fundamental family of eigenfunctions of $\Delta_{\nu}$ with eigenvalue $-((\nu-n)^2+\lambda^2)$ was given by the following family of Poisson kernels
\begin{align}\label{E2.15}
P^{\nu}_{\lambda}(z,\omega)=(\frac{1-\mid z\mid^{2}}{\mid 1-<z,\omega>\mid^{2}})^\frac{i\lambda+n-\nu}{2}(1-<z,\omega>)^{-\nu},\hspace{0.2cm}z\in\mathbb{B}_{n}, \hspace{0.2cm}\omega\in \partial \mathbb{B}_{n}
\end{align}
representing the function $e_{\lambda,w}(z)$ in \cite{Zha}. Then the Fourier-Helgason transform (the generalized Fourier transform) is defined  by
\begin{align}\label{E2.16}
\tilde{F}(\lambda,\omega)=\int_{\mathbb{B}_{n}}F(z){P^{\nu}_{-\lambda}(z,\omega)}d\mu_{\nu}(z),
\end{align}
for $F$ in $C_{0}^{\infty}(\mathbb{B}_{n})$, the space of $C^{\infty}$-functions on $\mathbb{B}_{n}$ with compact support. Note that $\tilde{F}(\lambda,\omega)$ is extended to an entire function of $\lambda$.\\
Recall that if we take $\alpha=\nu-n-1$ in Theorem 2 given in \cite{Zha}, the condition $\nu>n$  and non integer lead us to write the following inversion formula for the Fourier-Helgason transform
\begin{align}\label{E2.17}
\nonumber F(z)&=\dfrac{1}{4}\frac{\Gamma(n)}{2^{2(\nu-n)}\pi^{n+1}} \int_{\partial\mathbb{B}_{n}}\int_{
\mathbb{R}}\tilde{F}(\lambda,\omega)P^{\nu}_{\lambda}(z,\omega)\mid C_{\nu}(\lambda)\mid^{-2}d\lambda d\sigma(\omega)\\
&+\sum_{j=0}^{[\frac{\nu-n}{2}]}c_{j}\int_{\partial\mathbb{B}_{n}}\tilde{F}(\lambda_{j},\omega)P^{\nu}_{\lambda_{j}}(z,\omega) d\sigma(\omega),
\end{align}
with $F\in C_{0}^{\infty}(\mathbb{B}_{n})$, $[\frac{\nu-n}{2}]$ is the integral part of $\frac{\nu-n}{2}$ and $c_{j}$ are the constants defined by equation (1.5) in  \cite{Zha} as
\begin{align}\label{E2.18}
c_{j}=\dfrac{\Gamma (n+j)\Gamma (k+1- j)\Gamma (\nu)(\nu -n -2j)}{\pi ^{n}\Gamma (n)\Gamma (\nu - n +1 -j )(-1)^{j}\displaystyle{\prod^{k}_{j^{'}=0,j^{'}\neq j}}(j^{'}-j)},
\end{align}
\begin{align}\label{E2.19}
\lambda_{j}=-i(\nu-n-2j)\hspace{0.2cm} \mbox{for}\hspace{0.2cm}j=0,...,[\frac{\nu-n}{2}].
\end{align}
Moreover, the operators
\begin{align}\label{E2.20}
P_j[F](z)=c_{j}\int_{\partial
\mathbb{B}_{n}}\tilde{F}(\lambda_{j},\omega)P^{\nu}_{\lambda_{j}}(z,\omega) d\sigma(\omega),
\end{align}
appearing in the discrete part of formula $(\ref{E2.17})$ can be extended to pairwise orthogonal projections on $L^{2}_{\nu}(\mathbb{B}_{n},d\mu_{\nu})$ (\cite{Zha}). In order to give a spectral interpretation of formula $(\ref{E2.17})$, we need the following Lemma.
\begin{lemma}\label{L2.1}\cite{Bou}
Let $\alpha$, $\beta$ be two real numbers and let $\Delta_{\alpha,\beta}$ be the following Laplacians
 \begin{align}\label{E2.21}
\Delta_{\alpha,\beta} =4(1-|z|^{2})\{\sum_{1\leq i,j\leq n}(\delta_{i,j}-z_{i}\overline{z}_{j})\frac{\partial^{2}}{\partial z_{i}\partial\overline{z}_{j}}+\alpha \sum_{j=1}^{n}z_{j}\frac{\partial}{\partial z_{j}}+\beta \sum_{j=1}^{n}\overline{z}_{j}\frac{\partial}{\partial\overline{z}_{j}}-\alpha\beta\}
\end{align}
acting on the Hilbert space $L^{2}(\mathbb{B}_{n},(1-\mid z \mid^2)^{-(\alpha+\beta+n)-1}dm(z))$. Then, for every $\gamma \in \mathbb{R}$, we have the following intertwining relation
\begin{align}\label{E2.22}
\Delta_{ \alpha,\beta}=M^{-1}\circ[\Delta_{\alpha-\gamma,\beta-\gamma}-4\gamma(\alpha+\beta+n-\gamma )]\circ M,
\end{align}
where $M$ is the unitary operator defined by
\begin{align}\label{E2.23}
  M :&L_{\alpha,\beta}^{2}(\mathbb{B}_{n} )\longrightarrow L_{\alpha-\gamma,\beta-\gamma}^{2}(\mathbb{B}_{n})\\
 \nonumber & F(z) \longmapsto (1-\mid z \mid ^{2})^{-\gamma}F(z).
\end{align}
\end{lemma}
Before given a spectral  interpretation of the discrete part involved in formula (\ref{E2.17}), we will give the following remark.
\begin{remark}\label{R2.1}
 The family of operators $\Delta_{\alpha ,\beta}$ (with slight modifications) was considered by many authors in several contexts   \cite{Bou,Pat,Int}. But in view point of $L^{2}$-spectral theory, the study of such operators can be reduced by means of the above lemma to the study of the operators $\Delta_{\nu}$ with $\nu=\alpha-\beta$. So, the consideration of the operators $\Delta_{\nu}$ is not restrictive.
\end{remark}
With the help of the above lemma, we can state the following proposition.
\begin{proposition}\label{P2.1}
 The operator $\Delta_{\nu}$ given in (\ref{E2.12}) and defined on the Hilbert space $L^{2}_{\nu}(\mathbb{B}_{n})$ with maximal domain
\begin{align}\label{E2.24}
\mathcal{D}=\{F\in L^{2}_{\nu}(\mathbb{B}_{n}),\hspace{0.2cm}\Delta_{\nu}F \in L^{2}_{\nu}(\mathbb{B}_{n})\},
\end{align}
is self-adjoint and its spectrum decomposes as
 \begin{align}\label{E2.25}
 \sigma(\Delta_{\nu})= \sigma_{p}(\Delta_{\nu}) \cup  \sigma_{c}(\Delta_{\nu}),
 \end{align}
 where the point spectrum $\sigma_{p}(\Delta_{\nu})$ is given by
\begin{align}\label{E2.26}
\sigma_{p}(\Delta_{\nu})=\{-(\lambda_{l}^{2}+(n-\nu)^{2}),\hspace{0.2cm}  \lambda_{l}=-i(\nu-n-2l)\hspace{0.2cm}\mbox{for}\hspace{0.2cm}l=0,...,[\frac{\nu-n}{2}]\}
\end{align}
and the continuous spectrum is as follows
\begin{align}\label{E2.27}
\sigma_{c}(\Delta_{\nu})=\{-(\lambda^{2}+(\nu-n)^{2}),\hspace{0.2cm}\lambda\in\mathbb{R}\}.
\end{align}
 \end{proposition}
 For the proof of this proposition we refer to the Proposition 2 in \cite{Int}. The reader can be see also the following references \cite{Els,Pat}. Now, we are in position to give a spectral interpretation for the discrete part of formula (\ref{E2.17}). Precisely, we have.
\begin{proposition}\label{P2.2}
Let $l \in \mathbb{Z}_{+}, 0 \leq l  <\dfrac{\nu-n}{2}$ and let the eigenspace
\begin{align}\label{E2.28}
\mathscr{E}^{\nu}_{\rho_{l}}(\mathbb{B}_{n})=\{F\in  L^{2}_{\nu}(\mathbb{B}_{n}),\hspace{0.2cm} \Delta_{\nu}F=\rho_{l}F\}
\end{align}
associated with the eigenvalue $\rho_{l}=-(\lambda_{l}^{2}+(n-\nu)^{2})$ with $\lambda_{l}=-i(\nu-n-2l)$. Then, we have
\begin{itemize}
\item[i)] The vector space $\mathscr{E}^{\nu}_{\rho_{l}}(\mathbb{B}_{n})$ is a closed subspace of $L_{\nu}^{2}(\mathbb{B}_{n})$.
\item[ii)] The orthogonal projector $\Pi_{l}$ on the subspace $\mathscr{E}^{\nu}_{\rho_{l}}(\mathbb{B}_{n})$ coincide with the projector $P_{l}$ defined in $(\ref{E2.20})$.
\item[iii)] The constant $c_{l}$ defined by  formula $(\ref{E2.18})$ can be also written as
\begin{align}\label{E2.29}
c_{l}=\frac{2\Gamma(n+l)}{\pi^{n}\Gamma(n)l!}\frac{(\nu-n-2l)\Gamma(\nu-l)}{\Gamma(\nu-n-l+1)}.
\end{align}
\end{itemize}
\end{proposition}
\begin{proof}
First, it is not difficult to see that the operator $\Delta_{\nu}$ is densely defined. Then, by using the well know relation between the range  $R(T)$ of a densely defined operator $T$ and the kernel of its adjoint $N(T^{*})=R(T)^{\bot}$ (see \cite[p.9]{Kon} for the general theory). Then, by using the self-adjointness of $\Delta_{\nu}$, we obtain the closeness of the space $\mathscr{E}^{\nu}_{\rho_{l}}(\mathbb{B}_{n})=N(\Delta_{\nu}-\rho_{l})$.  Note that for the operator $P_{l}$ defined in (\ref{E2.20}) it  is proved in \cite{Zha} that is an integral operator with the kernel
\begin{align}\label{E2.30}
K_{l}^{\nu}(z,w)=c_{l}(1-<z,w>)^{-\nu}\prescript{}{2}{F}_1^{}\left( -l,l-\nu +n ; n; 1- \dfrac{\vert 1-<z,w>\vert^{2}}{(1-\vert z\vert^{2})(1-\vert w \vert^{2})}\right) ,
\end{align}
and  $c_{l}$ is the constant given in equation (\ref{E2.18}). Then, by the general theory,  the kernel  $K_{l}^{\nu}(z,w)$ is also the reproducing kernel of the subspace
\begin{align}\label{E2.31}
A^{\nu,2}_{l}(\mathbb{B}_{n}):=P_{l}(L^{2}_{\nu}(\mathbb{B}_{n})).
\end{align}
In other hand, by taking $\alpha=0$, $\beta=-\nu$ and $\gamma =\dfrac{-\nu}{2}$ in formula (\ref{E2.22}) of Lemma (\ref{L2.1}), we obtain the following relation
\begin{align}\label{E2.32}
\Delta_{\nu}=(1-\mid z\mid^{2})^{-\nu/2}(\Delta_{\nu /2,-\nu /2}+2\nu(n-\dfrac{\nu}{2})I)(1-\mid z\mid^{2})^{\nu /2}
\end{align}
with
\begin{align}\label{E2.33}
 \Delta_{\nu/2,-\nu/2}=4(1-|z|^{2})\{\sum_{1\leq i,j\leq n}(\delta_{i,j}-z_{i}\overline{z}_{j})\frac{\partial^{2}}{\partial z_{i}\partial\overline{z}_{j}}+\dfrac{\nu }{2} \sum_{j=1}^{n}(z_{j}\frac{\partial}{\partial z_{j}}- \overline{z}_{j}\frac{\partial}{\partial\overline{z}_{j}})+\dfrac{\nu ^2}{4} \}.
\end{align}
Remark that the operator $\Delta_{\nu/2,-\nu/2}$ can be written in terms of the operator $H_{B,\rho}$, defined in \cite[p.2]{Gha} for $B=\dfrac{\nu}{2}$ and $\rho=1$, as follows
\begin{align}\label{E2.34}
\Delta_{\nu/2,-\nu/2}=-H_{\nu/2,1}+\nu^{2},
\end{align}
with
 \begin{align}\label{E2.35}
  H_{\nu/2,1}=-4(1-|z|^{2})\{\sum_{1\leq i,j\leq n}(\delta_{i,j}-z_{i}\overline{z}_{j})\frac{\partial^{2}}{\partial z_{i}\partial\overline{z}_{j}}+\dfrac{\nu }{2} \sum_{j=1}^{n}(z_{j}\frac{\partial}{\partial z_{j}}- \overline{z}_{j}\frac{\partial}{\partial\overline{z}_{j}}) \}+\nu^{2}\vert z \vert^{2}.
\end{align}
Then, the relation (\ref{E2.32}) can be rewritten as
\begin{align}\label{E2.36}
\Delta_{\nu}F=\left(1-\mid z\mid^{2}\right) ^{-\nu  /2}(-H_{\nu /2,1}+2\nu nI)(1-\mid z \mid^{2})^{\nu /2}F, \,\, F\in \mathcal{D}\subset L^{2}_{\nu}(\mathbb{B}_{n}).
\end{align}
Recall that from \cite[p.2]{Gha} the point spectrum of the operator $H_{\nu/2,1}$ acting on the Hilbert space
$L^{2}_{\nu}(\mathbb{B}_{n},(1-\vert z ^{2})^{-n-1}dm(z))$ is given by
 \begin{align}\label{E2.37}
\sigma_{p}(H_{\nu/2,1}) =\lbrace q_{l}=2\nu(2l+n)-4l(l+n),\hspace{0.2cm} l\in \mathbb{Z}_{+}, \,\,0 \leq l< \dfrac{\nu-n}{2}\rbrace.
\end{align}
Also, recall that the eigenspace $\mathscr{E}^{\nu}_{l}(\mathbb{B}_{n})$ corresponding to the eigenvalue\\
 $\rho_{l}=-(\lambda _{l}^{2}+(n-\nu)^{2})=4l(l+n-\nu)$ is a  closed subspace of $L^{2}_{\nu}(\mathbb{B}_{n})$. Let us denote by $H_{l}^{\nu}(z,w)$ the reproducing kernel of the space $\mathscr{E}^{\nu}_{l}(\mathbb{B}_{n})$. Our aims in this part is to compute this reproducing kernel. To this end, we consider $F$ in $\mathscr{E}^{\nu}_{l}(\mathbb{B}_{n})$, then from relation (\ref{E2.36}) we have  the following formula
 \begin{align}\label{E2.38}
H_{\nu /2 ,1}[(1-\vert z\vert ^{2})^{\nu/2}F]=\left( 2\nu n-\rho_{l}\right) (1-\vert z\vert ^{2})^{\nu/2}F.
\end{align}
Remark that $2\nu n-\rho_{l}=q_{l}$ and $(1-\vert z\vert^{2} )^{\nu/2}F \in L_{\nu}^{2}(\mathbb{B}_{n},(1-\vert z\vert ^{2})^{-n-1}dm(z))$. Then, the function $(1-\vert z\vert^{2} )^{\nu/2}F$  belongs to the  eigenspace of the operator $H_{\nu/2,1}$ with $q_{l}$ as the associated eigenvalue. Recall that, by (i) of Proposition 7 in \cite{Gha} the reproducing kernel of this last eigenspace is given by the formula
\begin{align}\label{E2.39}
K_{l,1}^{\nu /2}(z,w)&=A_{n,\nu /2,l}(1) \left(\dfrac{\overline{1-<z,w>}}{1-<z,w>}\right)^{\nu /2}\left(\dfrac{(1-\vert z \vert ^2)(1-\vert w \vert ^2)}{\vert 1- <z,w>  \vert^2}\right)^{\nu /2-l}\\
\nonumber\\
\nonumber &\times .{_2}F_1\left(-l,\nu -l,n,1-\dfrac{(1-\vert z \vert ^2)(1-\vert w \vert ^2)}{\vert 1- <z,w>  \vert^2}\right),
\end{align}
with
\begin{align}\label{E2.40}
A_{n,\nu /2,l}(1)=\frac{1}{\pi^{n}}\frac{\Gamma (l+n)}{\Gamma(n)l!}\frac{(\nu-n-2l)\Gamma(\nu-l)}{\Gamma(\nu-n-l+1)}.
\end{align}
Then for the function $ (1-\vert z\vert ^{2})^{\nu/2}F $, we can  write the reproducing formula
\begin{align}\label{E2.41}
(1-\vert z\vert ^{2})^{\nu/2}F(z)=\int_{\mathbb{B}_{n}} K_{l,1}^{\nu /2}(z,w)(1-\vert w \vert^{2})^{\nu /2}F(w)(1-\vert w \vert^{2})^{-n-1}dm(w),
\end{align}  its follows that the reproducing kernel $H_{l}^{\nu }(z,w) $ of the eigenspace $\mathscr{E}_{l}(\mathbb{B}_{n})$ is given by
\begin{align}\label{E2.42}
H_{l}^{\nu }(z,w)=(1-\vert z\vert ^{2})^{-\nu/2}K_{l,1}^{\nu /2}(z,w)(1-\vert w\vert ^{2})^{-\nu/2}.
\end{align}
From the expression of the kernel $K_{l,1}^{\nu /2}(z,w)$ given in (\ref{E2.39}), it follows that
 \begin{align}\label{E2.43}
\nonumber H^{\nu }_{l}(z,w)&=(1-\vert z\vert ^{2})^{-\nu/2}(1-\vert w\vert ^{2})^{-\nu/2} A_{n,\nu /2,l}(1)\left(\dfrac{\overline{1-<z,w>}}{1-<z,w>}\right)^{\nu /2}\left(\dfrac{(1-\vert z \vert ^2)(1-\vert w \vert ^2)}{\vert 1- <z,w>  \vert^2}\right)^{\nu /2-l}\\
\nonumber &\times {_2}F_1\left(-l,\nu -l;n;1-\dfrac{(1-\vert z \vert ^2)(1-\vert w \vert ^2)}{\vert 1- <z,w> \vert^2}\right)\\
\nonumber&=A_{n,\nu /2,l}(1)(1-<z,w>)^{-\nu}\left(\frac{|1-<z,w>|^{2}}{(1-|z|^{2})(1-|w|^{2})}\right)^{l}\\
&\times {_2}F_1\left(-l,\nu -l;n;1-\dfrac{(1-\vert z \vert ^2)(1-\vert w \vert ^2)}{\vert 1- <z,w> \vert^2}\right).
\end{align}
By using the change of variable
\begin{align}\label{E2.44}
1-\dfrac{(1-\vert z \vert ^2)(1-\vert w \vert ^2)}{\vert 1- <z,w>  \vert^2}=\dfrac{x}{x-1},
\end{align}
the equation (\ref{E2.43}) becomes
\begin{align}\label{E2.45}
H^{\nu }_{l}(z,w)= A_{n,\nu /2,l}(1)\left (1-<z,w>\right )^{-\nu }(1-x)^{l}{_2}F_1\left(-l,\nu -l;n;\dfrac{x}{x-1}\right).
\end{align}
Applying  the following formula \cite[p.47]{Mag}
\begin{align}\label{E2.46}
{_2}F_1\left(a,b; c;z \right)=\left(1-z \right)^{-a}{_2}F_1\left(a,c-b; c;\dfrac{z}{z-1}\right).
\end{align}
For $a=-l$, $b=n+l-\nu$ and $c=n$, the equation (\ref{E2.45}) can be rewritten as
\begin{align}\label{E2.47}
H^{\nu }_{l}(z,w)= A_{n,\nu /2,l}(1)\left (1-<z,w>\right )^{-\nu }{_2}F_1\left(-l,n+l -\nu ;n;1-\dfrac{\vert 1- <z,w>  \vert^2}{(1-\vert z \vert ^2)(1-\vert w \vert ^2)}\right),
\end{align}
where $A_{n,\nu /2,l}(1)$ is the constant defined in (\ref{E2.40}). Taking into account  that the function $H_{l}^{\nu}(z,w)$ defined in (\ref{E2.43}) which is the reproducing kernel of the space $\mathscr{E}_{l}^{\nu}(\mathbb{B}_{n})$ can be rewritten in terms of the function $K_{l}^{\nu}(z,w)$ defined in (\ref{E2.30}) as follows
\begin{align}\label{E2.48}
H_{l}^{\nu}(z,w)=\kappa K_{l}^{\nu}(z,w),
\end{align}
where the constant $\kappa$ is give by
\begin{align}\label{E2.49}
\kappa=\dfrac{A_{n,\nu /2,l}(1)}{c_{l}}.
\end{align}
Is not hard to see that $\kappa>0$, then  we can consider on the space $\mathscr{E}_{l}^{\nu}(\mathbb{B}_{n})$  the following new scalar product
\begin{align}\label{E2.50}
 <f,g>_{\kappa}=\kappa <f,g>,\hspace{0.2cm}\mbox{for}\hspace{0.2cm}f,\hspace{0.2cm}g\in\mathscr{E}_{l}^{\nu}(\mathbb{B}_{n}).
\end{align}
It is easy to see that the space $(\mathscr{E}_{l}^{\nu}(\mathbb{B}_{n}),< >_{\kappa})$ is a reproducing kernel Hilbert space which having the function $ K_{l}^{\nu}(z,w) $ as reproducing kernel. Then, by using Proposition 2.3 in \cite[p.18]{Pau}, we deduce that
\begin{align}\label{E2.51}
A_{l}^{\nu ,2}(\mathbb{B}_{n})=\mathscr{E}_{l}^{\nu}(\mathbb{B}_{n})\hspace{0.2cm}\mbox{and}\hspace{0.2cm}\parallel f \parallel_{\kappa} = \parallel f \parallel.
\end{align}
Then, we obtain
\begin{align}\label{E2.52}
 \kappa \parallel f \parallel=\parallel f \parallel,\hspace{0.2cm}\mbox{for}\hspace{0.2cm}\mbox{all}\hspace{0.2cm}
  f\in\mathscr{E}_{l}^{\nu}(\mathbb{B}_{n}).
\end{align}
It follows that $\kappa =1$, so we obtain $c_{l}=A_{n,\nu /2,l}(1) $ and thus the projector $P_{l}$ on $ A_{l}^{\nu ,2}(\mathbb{B}_{n}) $ is also the projector on the eigenspace $\mathscr{E}_{l}^{\nu}(\mathbb{B}_{n})$. Hence the discrete part in formula (\ref{E2.5}) correspond to the orthogonal projection of a function $f \in C_{0}^{\infty}(B_{n})$ on the orthogonal direct sum
\begin{align}\label{E2.53}
\displaystyle{\bigoplus_{0\leq l <\frac{\nu -n}{2}}}\mathscr{E}_{l}^{\nu}(\mathbb{B}_{n}).
\end{align}
Now, for a spectral interpretation of the continuous part of formula (\ref{E2.17}),  is not hard to see that the involved integral  can be written as an integral with respect to the parameter $\lambda $ over the set $\mathbb{R}$ of the following $C^{\infty}-$eigenfunctions of the operator $\Delta_{\nu},$
\begin{align}\label{E2.54}
G_{\lambda}(z)=\dfrac{1}{4}\frac{\Gamma(n)}{2^{2(\nu-n)}\pi^{n+1}}\mid C_{\nu}(\lambda)\mid^{-2} \int_{
\partial\mathbb{B}_{n}}\tilde{F}(\lambda,\omega)P^{\nu}_{\lambda}(z,\omega)d\lambda d\sigma(\omega)
\end{align}
corresponding to the eigenvalues $-(\lambda^{2}+(n-\nu)^{2})$.
\end{proof}
\begin{remark}\label{R2.2}
Notice that the continuous spectrum $\sigma_{c}(\Delta_{\nu})=\{-(\lambda^{2}+(\nu-n)^{2}),\hspace{0.25cm}\lambda\in \mathbb{R}\}$ of the operator  $\Delta_{\nu}$ is understood is the sense that $z\in \sigma_{c}(\Delta_{\nu})$  if and only if the rang of the operator $(\Delta_{\nu}-zI)$  is not closed. Thus, is not surprising that the sets $\sigma_{c}(\Delta_{\nu})$ and $\sigma_{p}(\Delta_{\nu})$  are not disjoint (see \cite[p.30]{Kon} for more general theory).
\end{remark}
                                              \section{Spectral density}
Our aim in this section  is to compute the Schwartz kernel for the spectral density of the $G-$invariant shifted Laplacian
\begin{align}\label{E3.1}
 \tilde{\Delta}_{\nu}=-(\Delta_{\nu}+(\nu-n)^{2}),
 \end{align}
with $\mathcal{D}$ its maximal domain.\\
Note that the spectrum of the operator $\tilde{\Delta}_{\nu}$ can be given easily from (\ref{E2.25}) by
\begin{align}\label{E3.2}
 \sigma(\tilde{\Delta}_{\nu})=\left\{s_{j}=\lambda_{j}^{2};\hspace{0.2cm}j=0,...,[\frac{\nu-n}{2}]\right\}\cup
\left\{s=\lambda^{2},\hspace{0.2cm}\lambda\in\mathbb{R}\right\},
 \end{align}
where $\lambda_{j}=-i(\nu-n-2j)$. It is easy to see that the  operator $\tilde{\Delta}_{\nu}$  considered on  the minimal domain $C^{\infty}_{0}(\mathbb{B}_{n})$ is symmetric. Then, by using  Proposition 1.14 in \cite[p.20]{Kon}, we obtain that $\tilde{\Delta}_{\nu}$ is essentially self-adjoint operator, and then the self-adjoint extension of $\tilde{\Delta}_{\nu}$  is $-(\Delta_{\nu}+(\nu-n)^{2})$ defined on the maximal domain $\mathcal{D}$ giving in $(\ref{E1.5})$. This extension admits a spectral decomposition \cite{Kon,Lei,Mor}
\begin{align}\label{E3.3}
I=\int_{-\infty}^{+\infty}dE^{\nu}_{s},\hspace{0.2cm} \tilde{\Delta}_{\nu}=\int_{-\infty}^{+\infty}sdE^{\nu}_{s}.
\end{align}
Then, in the weak sense, we have
\begin{align}\label{E3.4}
(f,g)=\int_{-\infty}^{+\infty}d(E^{\nu}_{s}f,g),\hspace{0.2cm} (\tilde{\Delta}_{\nu}f,g)=\int_{-\infty}^{+\infty}sd(E^{\nu}_{s}f,g),
\hspace{0.2cm} \mbox{for}\hspace{0.2cm}f\in C^{\infty}_{0}(\mathbb{B}_{n})\hspace{0.2cm}\mbox{and}\hspace{0.2cm}g\in \mathrm{L}_{\nu}^{2}(\mathbb{B}_{n}).
\end{align}
The spectral density \cite{Est}
\begin{align}\label{E3.5}
e^{\nu}_{s}:=\frac{dE^{\nu}_{s}}{ds}
\end{align}
is understood as an operator-valued distribution, an element of the space\\
$\mathcal{D'}(\mathbb{R}, L(\mathcal{D},\mathrm{L}_{\nu}^{2}(\mathbb{B}_{n})))$, where $L(\mathcal{D},\mathrm{L}_{\nu}^{2}(\mathbb{B}_{n}))$ is the space of bounded operators from $\mathcal{D}$ to $\mathrm{L}_{\nu}^{2}(\mathbb{B}_{n}).$
In terms of the spectral density $e^{\nu}_{s}=\frac{dE^{\nu}_{s}}{ds}$, the two equations given in (\ref{E3.3}) become as follows, respectively.
\begin{align}\label{E3.6}
I=<e^{\nu}_{s},1>,\hspace{0.2cm}\tilde{\Delta}_{\nu}=<e^{\nu}_{s},s>,
\end{align}
where $<f(s),\phi(s)>$ is the evaluation of the distribution $f(s)$ on a test function $\phi(s)$.
Since $\tilde{\Delta}_{\nu}$ is an elliptic operator (see \cite{Fol} for the general theory), then its spectral density $e^{\nu}_{s}$ admits a distributional kernel $e(s,z,w)$ \cite{Est}  an element of $\mathcal{D}'(\mathbb{R},\mathcal{D}'(\mathbb{B}_{n}\times \mathbb{B}_{n}))$. Precisely, we have the following proposition.
\begin{proposition}\label{P3.1} The Schwartz kernel $e^\nu(s,w,z)$ of the spectral density $e^{\nu}_{s}=\frac{dE^{\nu}_{s}}{ds}$ is given by
\begin{align}\label{E3.7}
\nonumber e^\nu(s,w,z)&=\frac{\Gamma(n)}{4\pi^{n+1}2^{2(\nu-n)}}(1-<z,w>)^{-\nu}|C_{\nu}(\sqrt{s})|^{-2}(\sqrt{s})^{-1}
\chi_{+}(s)\phi_{\sqrt{s}}^{(n-1,-\nu)}(d(z,w))\\
&+\sum_{j=0}^{[\frac{\nu-n}{2}]}c_{j}\frac{j!}{(n)_{j}}(1-<z,w>)^{-\nu}P_{j}^{(n-1,-\nu)}(\cosh(2d(z,w)))\delta(s-s_{j}),
\end{align}
with
\begin{align}\label{E3.8}
 s_{j}=-(2j+\nu-n)^{2};j=0,...,[\frac{\nu-n}{2}],
 \end{align}
\begin{align}\label{E3.9}
c_{j}=\frac{2\Gamma(n+j)}{\pi^{n}\Gamma(n)j!}\frac{(\nu-n-2j)\Gamma(\nu-j)}{\Gamma(\nu-n-j+1)},
\end{align}
$\chi_{+}(s)$ is the characteristic function of the set of positive real numbers and
\begin{align}\label{E3.10}
\phi^{(\alpha,\beta)}_{\lambda}(t)=\prescript{}{2}{F}_1^{}\left( \frac{\alpha+\beta+1-i\lambda}{2},\frac{\alpha+\beta+1+i\lambda}{2},1+\alpha;-\sinh^{2}t\right),
\end{align}
 is the Jacobi function (\cite[p.5]{Koo}).  $P^{(\alpha,\beta)}_j(x)$ denote the classical Jacobi polynomial of degree $j$ \cite{Bea} and $d(z,w)$ is the distance  coresponding to the Bergmann  metric in the unit ball $\mathbb{B}_{n}$ (\cite[p.25]{Zhu})
\begin{align}\label{E3.11}
\cosh^{2} d(z,w)=\frac{\mid 1-<z,w>\mid^{2}}{(1-\mid z\mid^{2})(1-\mid w\mid^{2})}, (z,w)\in\mathbb{B}_{n}\times\mathbb{B}_{n},
\end{align} The function
\begin{align}\label{E3.12}
C_{\nu}(\lambda)=\dfrac{2^{-\nu+n-i\lambda}\Gamma(n)\Gamma(i\lambda)}{\Gamma(\frac{i\lambda+n-\nu}{2})
\Gamma(\frac{i\lambda+n+\nu}{2})},
\end{align}
is the analogous of the Harish-Chandra C-function.
\end{proposition}
For given the proof of the above proposition, we need the following Lemma
\begin{lemma}\label{L3.1}We have the following formula
\begin{align}\label{E3.13}
\int_{\partial\mathbb{B}_{n}}P^{\nu}_{\lambda}(z,\omega)\overline{P^{\nu}_{\lambda}(w,\omega)}d\sigma(\omega)
=(1-<z,w>)^{-\nu}\prescript{}{2}{F}_1^{}\left( \frac{i\lambda +n -\nu}{2},\frac{-i\lambda +n -\nu}{2}, n;-\sinh^{2}(d(z,w))\right).
\end{align}
\end{lemma}
\begin{proof}
First, we  will show that the function
\begin{align}\label{E3.14}
K_{\lambda}^{\nu}(z,w)=\int_{\partial\mathbb{B}_{n}}P^{\nu}_{\lambda}(z,\omega)
\overline{P^{\nu}_{\lambda}(w,\omega)}d\sigma(\omega)
\end{align}
satisfies the following property
\begin{align}\label{E3.15}
K_{\lambda}^{\nu}(g.z,g.w)=(Cz+D)^{\nu}\overline{(Cw+D)}^{\nu}K_{\lambda}^{\nu}(z,w),
\hspace{0.2cm}\mbox{for}\hspace{0.2cm}
g=\left(
\begin{matrix}
A & B \\
C & D
\end{matrix}
\right)\in SU(1,n).
\end{align}
By using the fact that
\begin{align}\label{E3.16}
1-<gz,gw>=\dfrac{1-<z,w>}{(Cz+D)\overline{(Cw+D)}},
\end{align}
we can easily see that
\begin{align}\label{E3.17}
P^{\nu}_{\lambda}(g.z,w)=P^{\nu}_{\lambda}(z,g^{-1}w)(Cz+D)^{\nu}\mid Cg^{-1}w+D\mid ^{i\lambda +n}\left( \dfrac{\mid Cg^{-1}w+D\mid}{\overline{Cg^{-1}w+D}}\right)^{-\nu}.
\end{align}
Then, we get
\begin{align}\label{E3.18}
\nonumber K_{\lambda}^{\nu}(g.z,g.w)&= \int_{\partial\mathbb{B}_{n}}P^{\nu}_{\lambda}(g.z,\omega)\overline{P^{\nu}_{\lambda}(g.w,\omega)}d\sigma(\omega)\\
&=(Cz+D)^{\nu}\overline{(Cw+D)^{\nu}}\int_{\partial\mathbb{B}_{n}}P^{\nu}_{\lambda}(z,g^{-1}.\omega)\overline{P^{\nu}_{\lambda}
(w,g^{-1}\omega)}\mid Cg^{-1}\omega +D\mid^{2n}d\sigma(\omega)\\
\nonumber &=(Cz+D)^{\nu}\overline{(Cw+D)^{\nu}}\int_{\partial\mathbb{B}_{n}}P^{\nu}_{\lambda_{j}}(z,\xi)\overline{P^{\nu}_{\lambda}(w,\xi)}d\sigma(\xi)
\end{align}
where we have set $\xi=g^{-1}w$ and $d\sigma(w)=\mid C \xi+D\mid^{-2n}d\sigma(\xi)$. Thus, we get
\begin{align}\label{E3.19}
K_{\lambda}^{\nu}(g.z,g.w)=(Cz+D)^{\nu}\overline{(Cw+D)}^{\nu}K_{\lambda}^{\nu}(z,w).
\end{align}
Let $g_{z}$ be the following matrix
\begin{align}\label{E3.20}
g_{z}=\left(
\begin{matrix}
A & B \\
C & D
\end{matrix}
\right):=\left(
  \begin{array}{cc}
    (I_{n}-zz^{*})^{-1/2} & z(1-\vert z\vert^{2} )^{-1/2} \\\\
    z^{*}(I_{n}-zz^{*})^{-1/2} &  (1-\vert z\vert^{2} )^{-1/2}\\
  \end{array}
\right).
\end{align}
Is not hard to show that $g_{z}$ belongs to the group $SU(1,n)$ and then  by formula (\ref{E2.6}),
 we get the inverse of $g_{z}$ as follows
\begin{align}\label{E3.21}
g_{z}^{-1}=\left(
  \begin{array}{cc}
    (I_{n}-zz^{*})^{-1/2} & -(I_{n}-zz^{*})^{-1/2}z  \\\\
    z^{*}(1-\vert z\vert^{2} )^{-1/2} & (1-\vert z\vert^{2} )^{-1/2}\\
  \end{array}
\right).
\end{align}
We have that the action of the matrix $g_{z}$ satisfy $g_{z}.0=z$. So, thank to this last relation combined with  formula (\ref{E3.19}), we can write
\begin{align}\label{E3.22}
\nonumber K_{\lambda}^{\nu}(z,w)&=K_{\lambda}^{\nu}(g_{z}.0,w)\\
\nonumber\\
\nonumber&=K_{\lambda}^{\nu}(g_{z}.0,g_{z}g_{z}^{-1}w)\\
\nonumber\\
 &=D^{\nu}\overline{(Cg_{z}^{-1}w+D)^{\nu}}K_{\lambda}^{\nu}(0,g_{z}^{-1}.w).
\end{align}
Observe that
\begin{align}\label{E3.23}
1-<z,w>&=1-<g_{z}0,g_{z}g_{z}^{-1}w>\\
\nonumber &=\dfrac{1}{D(\overline{Cg^{-1}_{z}w+D})},
\end{align}
where we have used formula (\ref{E3.16}). Thus, the formula (\ref{E3.22}) can be rewritten as
\begin{align}\label{E3.24}
K_{\lambda}^{\nu}(z,w)=(1-<z,w>)^{-\nu}K_{\lambda}^{\nu}(0,g^{-1}_{z}w).
\end{align}
By setting $\xi=g^{-1}_{z}w$, the above equation becomes
\begin{align}\label{E3.25}
K_{\lambda}^{\nu}(z,w)=(1-<z,w>)^{-\nu}K_{\lambda}^{\nu}(0,\xi).
\end{align}
By using the formula \cite{Zha}
\begin{align}\label{E3.26}
\nonumber \Phi_{\lambda}(z) &= \int_{\partial\mathbb{B}_{n}}P^{\nu}_{\lambda}(z,\omega)d\sigma(\omega) \\
&=\left( 1-\mid z \mid ^{2}\right)^{\dfrac{-\nu+n-i\lambda}{2}}\prescript{}{2}{F}_1^{}\left( \frac{-i\lambda +n +\nu}{2},\frac{-i\lambda +n -\nu}{2}, n;\mid z \mid^{2}\right).
\end{align}
The equation (\ref{E3.25}) can be also rewritten as
\begin{align}\label{E3.27}
\nonumber K_{\lambda}^{\nu}(z,w)&=(1-<z,w>)^{-\nu}\int_{\partial\mathbb{B}_{n}}\overline{P^{\nu}_{\lambda}(\xi,\omega)}d\sigma(\omega) \\
\nonumber &=(1-<z,w>)^{-\nu}\overline{\Phi_{\lambda}(\xi)} \\
&=(1-\mid \xi \mid ^{2})^{\dfrac{-\nu+n+i\lambda}{2}}(1-<z,w>)^{-\nu}\prescript{}{2}{F}_1^{}\left(\frac{i\lambda+n +\nu}{2},\frac{i\lambda +n -\nu}{2}, n;\mid \xi \mid^{2}\right).
\end{align}
 Not that
\begin{align}\label{E3.28}
\nonumber 1-|\xi|^{2}=1-\mid g ^{-1}_{z}w\mid^{2}&=\frac{1-\mid w \mid ^{2}}{\vert\left( 1-\mid z \mid ^{2}\right)^{-1/2}(1- z^{*}w) \vert^{2}}\\
& =\frac{(1-\mid w \mid ^{2})(1-\mid z \mid ^{2})}{\mid 1-<z,w>\mid^{2}}.
\end{align}
Hence, by using formula $(\ref{E3.28})$, we get
\begin{align}\label{E3.29}
\nonumber K_{\lambda}^{\nu}(z,w) &= (1-<z,w>)^{-\nu} \left[ \frac{(1-\mid w \mid ^{2})(1-\mid z \mid )^{2}}{\mid 1-<z,w>\mid^{2}} \right] ^{\dfrac{i\lambda + n-\nu}{2}}\\
&\times\prescript{}{2}{F}_1^{}\left(\frac{i\lambda +n -\nu}{2},\frac{i\lambda +n+\nu}{2}, n;\mid g_{z}^{-1}w \mid^{2} \right).
\end{align}
Then, by using $(\ref{E3.28})$ and $(\ref{E2.46})$, we get
\begin{align}\label{E3.30}
\nonumber K_{\lambda}^{\nu}(z,w)&=(1-<z,w>)^{-\nu}\prescript{}{2}{F}_1^{}\left(\frac{i\lambda +n -\nu}{2},\frac{-i\lambda +n -\nu}{2}, n;\frac{\mid g_{z}^{-1}w \mid^{2}}{\mid g_{z}^{-1}z \mid^{2}-1}\right)\\
&= (1-<z,w>)^{-\nu}\prescript{}{2}{F}_1^{}\left(\frac{i\lambda +n-\nu}{2},\frac{-i\lambda +n -\nu}{2}, n; 1-\frac{\mid 1-<z,w>\mid^{2}}{(1-\mid z \mid ^{2})(1-\mid w \mid ^{2})}\right),
\end{align}
then by $(\ref{E1.16})$, we obtain
\begin{align}\label{E3.31}
 K_{\lambda}^{\nu}(z,w)=(1-<z,w>)^{-\nu}\prescript{}{2}{F}_1^{}\left(\frac{i\lambda +n -\nu}{2},\frac{-i\lambda +n -\nu}{2}, n; -\sinh^{2}(d(z,w))) \right).
\end{align}
Hence, the above equality is the desired result.
\end{proof}
\begin{proof}(\textbf{of Proposition} \ref{P3.1})\\
Let $F\in\mathcal{C}_{0}^{\infty}(\mathbb{B}_{n})$, then by using formula (\ref{E2.17}), we have
\begin{align}\label{E3.32}
\nonumber F(z)&=\dfrac{1}{4}\frac{\Gamma(n)}{2^{2(\nu-n)}\pi^{n+1}} \int_{
\partial\mathbb{B}_{n}}\int_{\mathbb{R}}\tilde{F}(\lambda,\omega)P^{\nu}_{\lambda}(z,\omega)\mid C_{\nu}(\lambda)\mid^{-2}d\lambda d\sigma(\omega)\\
\nonumber&+\sum_{j=0}^{k}c_{j}\int_{\mathbb{B}_{n}}\tilde{F}(\lambda_{j},\omega)P^{\nu}_{\lambda_{j}}(z,\omega) d\sigma(\omega)\\
\nonumber&=\frac{\Gamma(n)}{4\pi^{n+1}2^{2(\nu-n)}}\int_{-\infty}^{+\infty}d\lambda |C_{\nu}(\lambda)|^{-2}\int_{\mathbb{B}_{n}}\left(\int_{\partial\mathbb{B}_{n}}P^{\nu}_{\lambda}(z,\omega)P^{\nu}_{-\lambda}(w,\omega)d\sigma(\omega)\right)F(w)d\mu_{\nu}(w)\\
&+\sum_{0}^{[\frac{\nu-n}{2}]}c_{j}\int_{\mathbb{B}_{n}}\left(\int_{\partial\mathbb{B}_{n}}P^{\nu}_{\lambda_{j}}(z,\omega)P^{\nu}_{-\lambda_{j}}(w,\omega)d\sigma(\omega)\right)F(w)d\mu_{\nu}(w).
\end{align}
Making use of formula $(\ref{E3.13})$ in where the hypergeometric function in the right hand side was replaced by the corresponding Jacobi function defined in $(\ref{E3.10})$, then  we get
\begin{align}\label{E3.33}
F(z)&=\frac{\Gamma(n)}{4\pi^{n+1}2^{2(\nu-n)}}\int_{-\infty}^{+\infty}
\left[|C_{\nu}(\lambda)|^{-2}\int_{\mathbb{B}_{n}}(1-<z,w>)^{-\nu}(\phi^{(n-1,-\nu)}_{\lambda}(d(z,w)))F(w)d\mu_{\nu}(w)\right]d\lambda
\nonumber\\
&+\sum_{0}^{[\frac{\nu-n}{2}]}c_{j}\int_{\mathbb{B}_{n}}(1-<z,w>)^{-\nu}(\phi^{(n-1,-\nu)}_{\lambda_{j}}(d(z,w)))F(w)d\mu_{\nu}(w).
\end{align}
It is not difficult to see that the function $\lambda\longmapsto \phi^{(n-1,-\nu)}_{\lambda}(d(z,w))$ is even. Then, the  equation (\ref{E3.33}) can be rewritten as
\begin{align}\label{E3.34}
\nonumber F(z)&=\frac{\Gamma(n)}{42^{2(\nu-n)}\pi^{n+1}}2\int_{0}^{+\infty}d\lambda |C_{\nu}(\lambda)|^{-2}\int_{\mathbb{B}_{n}}(1-<z,w>)^{-\nu}(\phi^{(n-1,-\nu)}_{\lambda}(d(z,w)))F(w)d\mu_{\nu}(w)\\
&+\sum_{0}^{[\frac{\nu-n}{2}]}c_{j}\int_{\mathbb{B}_{n}}(1-<z,w>)^{-\nu}(\phi^{(n-1,-\nu)}_{\lambda_{j}}
(d(z,w)))F(w)d\mu_{\nu}(w).
\end{align}
Making use of the change of variable $s=\lambda^{2}$ in the first integral of (\ref{E3.34}) and setting
$s_{j}=\lambda_{j}^{2}=-(\nu-n-2j)^{2}$ in the discreet part of (\ref{E3.34}), we obtain
\begin{align}\label{E3.35}
\nonumber F(z)&=\frac{\Gamma(n)}{4\pi^{n+1}2^{2(\nu-n)}}\int_{0}^{+\infty}\left[|C_{\nu}(\sqrt{s})|^{-2}s^{\frac{-1}{2}}
\int_{\mathbb{B}_{n}}(1-<z,w>)^{-\nu}(\phi^{(n-1,-\nu)}_{\sqrt{s}}(d(z,w)))F(w)d\mu_{\nu}(w)\right]ds\\
&+\sum_{0}^{[\frac{\nu-n}{2}]}c_{j}\int_{\mathbb{B}_{n}}(1-<z,w>)^{-\nu}(\phi^{(n-1,-\nu)}_{\sqrt{s_{j}}}(d(z,w)))F(w)d\mu_{\nu}(w).
\end{align}
Using the Dirac notation
\begin{align}\label{E3.36}
\int_{-\infty}^{+\infty}\delta(s-s_{j})ds=1.
\end{align}
Then, the equation  $(\ref{E3.35})$ can be written in the distributional sense as follows
\begin{align}\label{E3.37}
F(z)=\int_{-\infty}^{\infty}1(\int_{\mathbb{B}_{n}}e(s,w,z)F(w)d\mu_{\nu}(w))ds,
\end{align}
where the Schwartz kernel $e^\nu (s,w,z)$ is given by
\begin{align}\label{E3.38}
\nonumber e^\nu (s,w,z)&=\frac{\Gamma(n)}{4\pi^{n+1}2^{2(\nu-n)}}(1-<z,w>)^{-\nu}\chi_{+}(s)|C_{\nu}(\sqrt{s})|^{-2}
(\sqrt{s})^{-1}\phi_{\sqrt{s}}^{(n-1,-\nu)}(d(z,w))\\
&+\sum_{j=0}^{[\frac{\nu-n}{2}]}c_{j}(1-<z,w>)^{-\nu}\phi_{\sqrt{s_{j}}}^{(n-1,-\nu)}(d(z,w)\delta(s-s_{j}).
\end{align}
Recall that $\lambda_{j}=\sqrt{s_{i}}:=-i(\nu-n-2j)$ for  $ j=0,1,...,\frac{n-\nu}{2}$. Then the Jacobi function $\phi_{\sqrt{s_{j}}}^{(n-1,-\nu)}(d(z,w))$ involved in the discreet part of the Schwartz kernel $e^{\nu}(s,w,z)$  becomes
\begin{align}\label{E3.39}
\phi_{\sqrt{s_{j}}}^{(n-1,-\nu)}(d(z,w))=\prescript{}{2}{F}_1^{}(-j,j+n-\nu,n;-\sinh^{2}(d(z,w)).
\end{align}
Next, by using the identity \cite[p.39]{Mag}
\begin{align}\label{E3.40}
P_{k}^{(\alpha,\beta)}(y)=\frac{(1+\alpha)_{k}}{k!}\prescript{}{2}{F}_1^{}(-k,\alpha+\beta+k+1,\alpha+1;\frac{1-y}{2}),
\end{align}
for $\alpha=n-1$, $\beta=-\nu$, $k=j$ and $y=1+2\sinh^{2}(d(z,w))=\cosh2d(z,w)$, the equation (\ref{E3.39}) becomes
\begin{align}\label{E3.41}
\phi_{\sqrt{s_{j}}}^{(n-1,-\nu)}(d(z,w))=\frac{j!}{(n)_{j}}P_{j}^{(n-1,-\nu)}(\cosh2d(z,w)).
\end{align}
Hence, we get
\begin{align}\label{E3.42}
\nonumber e^\nu (s,w,z)&=\frac{\Gamma(n)}{4\pi^{n+1}2^{2(\nu-n)}}(1-<z,w>)^{-\nu}\chi_{+}(s)|C_{\nu}(\sqrt{s})|^{-2}
(\sqrt{s})^{-1}\phi_{\sqrt{s}}^{(n-1,-\nu)}(d(z,w))\\
&+\sum_{j=0}^{[\frac{\nu-n}{2}]}c_{j}\frac{j!}{(n)_{j}}(1-<z,w>)^{-\nu}P_{j}^{(n-1,-\nu)}(\cosh(2d(z,w)))\delta(s-s_{j}).
\end{align}
Now, by returning back to equation (\ref{E3.32}) and applying the operator $\tilde{\Delta}_{\nu}$ to its both sides with the use of the fact that the involved function in the integral is an eigenfunction of $\tilde{\Delta}_{\nu}$ associated with the eigenvalues $s=\lambda^{2}$, $\lambda\in \mathbb{R}$. Also, we use the fact that each function given in the discrete part is an eigenfunction of $\tilde{\Delta}_{\nu}$ associated with the eigenvalue $s_{j}=\lambda_{j}^{2}$ for $j=0,1,...,[\frac{\nu-n}{2}]$. Then, by following the same method as in the above computation, we get the following  equality
\begin{align}\label{E3.43}
\tilde{\Delta}_{\nu}[F](z)=\int_{-\infty}^{+\infty}s\left[\int_{\mathbb{B}_{n}}e^{\nu}(s,w,z)F(w)d\mu_{\nu}(w)\right]ds.
\end{align}
Now, considering the functional $T$ which corresponds to a test function $\varphi$ the operator $<T,\varphi> \in \mathbf{L}(D,L_{\nu}^{2}(\mathbb{B}_{n}))$ defined by:
\begin{align}\label{E3.44}
\nonumber&<T,\varphi>[F](z)=\int_{-\infty}^{+\infty}\varphi(s)\left[\int_{\mathbb{B}_{n}}e^{\nu}(s,w,z)F(w)d\mu_{\nu}(w)\right]ds\\
\nonumber&:=\frac{\Gamma(n)}{4\pi^{n+1}2^{2(\nu-n)}}\int_{0}^{+\infty}|C_{\nu}(\sqrt{s})|^{-2}
(\sqrt{s})^{-1}\int_{\mathbb{B}_{n}}(1-<z,w>)^{-\nu}\phi_{\sqrt{s}}^{(n-1,-\nu)}(d(z,w))F(w)d\mu_{\nu}(w)\varphi(s)ds\\ &+\sum_{j=0}^{[\frac{\nu-n}{2}]}\varphi(s_{j})c_{j}\frac{j!}{(n)_{j}}\int_{\mathbb{B}_{n}}(1-<z,w>)^{-\nu}P_{j}^{(n-1,-\nu)}
(\cosh(2d(z,w)))F(w)d\mu_{\nu}(w).
\end{align}
Notice that from the above computation, is not hard to see that the projector $P_{j}$, defined in (\ref{E2.20}), can be written also in the following form
\begin{align}\label{E3.45}
P_{j}[F](z)=c_{j}\frac{j!}{(n)_{j}}\int_{\mathbb{B}_{n}}(1-<z,w>)^{-\nu}P_{j}^{(n-1,-\nu)}(\cosh(2d(z,w)))
F(w)d\mu_{\nu}(w).
\end{align}
We observe that the discrete part of the functional $<T,\varphi>[F]$ can be rewritten also as
\begin{align}\label{E3.46}
\sum_{j=0}^{[\frac{\nu-n}{2}]} \varphi(s_{j})P_{j}[F],
\end{align}
where $P_{j}$ is the orthogonal projector on the eigenspace defined in equation (\ref{E2.20}). Using the fact that the norm of a projector is equal to one. The norm of the discrete part satisfy the following estimate
\begin{align}\label{E3.47}
\parallel \sum_{j=0}^{[\frac{\nu-n}{2}]} \varphi(s_{j})P_{j}[F]\parallel_{L^{2}_{\nu}(\mathbb{B}_{n})}\leq
\parallel\varphi\parallel_{\infty}(1+\frac{\nu-n}{2})\parallel F\parallel_{L^{2}_{\nu}(\mathbb{B}_{n})},
\end{align}
with $\parallel\varphi\parallel_{\infty}=\displaystyle{\sup_{x\in\mathbb{R}}}\mid\varphi(x)\mid<\infty$.\\
Then, for proving that the functional $T$ is well defined, it is remain to prove that the following integral transform
\begin{align}\label{E3.48}
\hspace{-0.2cm}A_{\varphi}[F](z):=M_{n,\nu}\int_{0}^{+\infty}|C_{\nu}(\sqrt{s})|^{-2}
(\sqrt{s})^{-1}\int_{\mathbb{B}_{n}}(1-<z,w>)^{-\nu}\phi_{\sqrt{s}}^{(n-1,-\nu)}(d(z,w))F(w)d\mu_{\nu}(w)\varphi(s)ds
\end{align}
define a bounded operator on $L^{2}_{\nu}(\mathbb{B}_{n})$, where $M_{n,\nu}=\frac{\Gamma(n)}{4\pi^{n+1}2^{2(\nu-n)}}$. To do so, we first consider the case $\varphi\equiv1$. Then, thanks to equation (\ref{E3.35}), the function $A_{1}[F](z)$ can be rewritten as
\begin{align}\label{E3.49}
\nonumber A_{1}[F](z)&=F(z)-\sum_{j=0}^{[\frac{\nu-n}{2}]}c_{j}\frac{j!}{(n)_{j}}
\int_{\mathbb{B}_{n}}(1-<z,w>)^{-\nu}P_{j}^{(n-1,-\nu)}(\cosh(2d(z,w)))F(w)d\mu_{\nu}(w)\\
&=[I-\sum_{j=0}^{[\frac{\nu-n}{2}]}P_{j}][F](z),
\end{align}
where $F\in C_{0}^{\infty}(\mathbb{B}_{n})$.\\
Then, we have the following estimate
\begin{align}\label{E3.50}
\parallel A_{1}[F]\parallel_{L^{2}_{\nu}(\mathbb{B}_{n})}\leq(2+\frac{\nu-n}{2})\parallel F\parallel_{L^{2}_{\nu}(\mathbb{B}_{n})},\hspace{0.2cm}F\in C_{0}^{\infty}(\mathbb{B}_{n}).
\end{align}
Now, for a test function $\varphi$, it is easy to show that
\begin{align}\label{E3.51}
\mid A_{\varphi}[F](z)\mid\leq\parallel\varphi\parallel_{\infty}\mid A_{1}[F](z)\mid.
\end{align}
This last inequality combined with estimate (\ref{E3.50}) implies that
\begin{align}\label{E3.52}
\parallel A_{\varphi}[F](z)\parallel_{L^{2}_{\nu}(\mathbb{B}_{n})}\leq\parallel\varphi\parallel_{\infty}(2+\frac{\nu-n}{2})\parallel F\parallel_{L^{2}_{\nu}(\mathbb{B}_{n})},
\end{align}
which prove the boundedness of the operator $A_{\varphi}$.\\
Finally, by  (\ref{E3.48}) and (\ref{E3.52}), we get
\begin{align}\label{E3.53}
\parallel <T,\varphi>[F]\parallel_{L^{2}_{\nu}(\mathbb{B}_{n})}\leq\parallel\varphi\parallel_{\infty}(3+\nu-n)\parallel F\parallel_{L^{2}_{\nu}(\mathbb{B}_{n})}.
\end{align}
Furthermore, the two equations given in (\ref{E3.3}) become as follows,  respectively:
\begin{align}\label{E3.54}
<T,1>=I, \hspace{0.2cm} <T,s>=\tilde{\Delta}_{\nu}.
\end{align}
By uniqueness of the spectral density associated to a self-adjoint operator, we conclude that the functional $T$ is nothing but the spectral density of the operator $\tilde{\Delta}_{\nu}.$ This ends the proof.
\end{proof}
\begin{remark}\label{R3.2} Note that the operator valued function
 \begin{align}\label{E3.56}
 &\mathbb{R}\longrightarrow (L_{\nu}^{2}(\mathcal{D},B_{n}),\parallel.\parallel_{\mathcal{L}})\\
 \nonumber &\lambda\longmapsto E_{\lambda}^{\nu}
 \end{align}
 is locally integrable, where $\parallel.\parallel_{\mathcal{L}}$ is the classical norm on the Banach space of bounded operators from the domain $\mathcal{D}$ onto the whole space $L^{2}_{\nu}(B_{n})$. Then by using the fact that the point spectrum
 \begin{align}\label{E3.57}
\sigma_{p}(\tilde{\Delta}_{\nu})=\{s_{j}=
\lambda_{j}^{2},\hspace{0.2cm}j=0,...,[\frac{\nu-n}{2}]\}
 \end{align}
 is the set of the discontinuity points of the vector valued function $s\longmapsto E_{s}^{\nu}$. Then the derivation formula $\frac{dE_{s}^{\nu}}{ds}$ in the distributional sense must contain the jumps $E^{\nu}_{s_{j}+0}-E^{\nu}_{s_{j}-0}$ and the Dirac distribution $\delta(s-s_{j})$. So, is not surprising that the discrete part in the Schwartz kernel  $e^{\nu}(s,z,w)$ involves the above Dirac distributions.
\end{remark}
Based on the explicit formula of the spectral density  $\frac{dE_{s}^{\nu}}{ds}$, we use the functional
calculus developed by Estrada and Fuling \cite{Est} to define, for a suitable function  $f:\mathbb{R}\rightarrow\mathbb{C}$, the operator function $f(\tilde{\Delta}_{\nu})$ as follows
\begin{align}\label{E3.58}
f(\tilde{\Delta}_{\nu})[\varphi](z)=\int_{\mathbb{B}_{n}}\Omega_{f}(w,z)\varphi(w)d\mu_{\nu}(w),
\end{align}
where the distributional kernel  $\Omega_{f}(w,z)$ is given by
\begin{align}\label{E3.59}
\nonumber\Omega_{f}(w,z)&=\int_{\sigma(\tilde{\Delta}_{\nu})}e^{\nu}(s,w,z)f(s)ds\\
\nonumber&:=\frac{\Gamma(n)}{4\pi^{n+1}2^{2(\nu-n)}}(1-<z,w>)^{-\nu}\int_{0}^{+\infty}|C_{\nu}(\sqrt{s})|^{-2}(\sqrt{s})^{-1}
\phi_{\sqrt{s}}^{(n-1,-\nu)}(d(z,w))f(s)ds\\
&+\sum_{j=0}^{[\frac{\nu-n}{2}]}c_{j}\frac{j!}{(n)_{j}}(1-<z,w>)^{-\nu}
P_{j}^{(n-1,-\nu)}(\cosh(2d(z,w)))f(s_{j}),
\end{align}
with $C_{\nu}$ is the Harish-Chandra function defined in $(\ref{E3.12})$, $P^{(n-1,-\nu)}_j(.)$ are the Jacobi polynomials \cite{Bea}, $\phi_{\sqrt{s}}^{(n-1,-\nu)}(.)$ is the Jacobi function defined in $(\ref{E3.10})$ and
\begin{align}\label{E3.60}
 s_{j}=-(2j+\nu-n)^{2};j=0,...,[\frac{\nu-n}{2}],
 \end{align}
\begin{align}\label{E3.61}
c_{j}=\frac{2\Gamma(n+j)}{\pi^{n}\Gamma(n)j!}\frac{(\nu-n-2j)\Gamma(\nu-j)}{\Gamma(\nu-n-j+1)}.
\end{align}
                           \section{Heat and resolvent kernels}
In this section, we solve the heat equation and  we give the resolvent kernel. First, we are interested to following Cauchy problem of heat equation associated with the operator $\Delta_{\nu}$ on $\mathbb{B}_{n}$:
\begin{align}\label{E4.1}
\left\{
  \begin{array}{ll}
    \frac{\partial u(t,z)}{\partial t}=\Delta_{\nu}u(t,z),\hspace{0.2cm} (t,z)\in\mathbb{R}^{+}\times \mathbb{B}_{n}, & \hbox{} \\\\
    u(0,z)=\varphi(z)\in\mathcal{C}_{0}^{\infty}(\mathbb{B}_{n}).& \hbox{}
  \end{array}
\right.
\end{align}
For the above Cauchy problem, we have the following proposition.
\begin{proposition}\label{P4.1} The solution $u(t,z)$ of the Cauchy problem $(\ref{E4.1})$ is given by the following  integral formula
\begin{align}\label{E4.2}
u(t,z)=\int_{\mathbb{B}_{n}}K_{\nu}(t,z,w)\varphi(w)d\mu_{\nu}(w),
\end{align}
where $K_{\nu}(t,z,w)$ is the heat kernel given as follows
\begin{align}\label{E4.3}
\nonumber K_{\nu}(t,z,w)&=(1-<z,w>)^{-\nu}\sum_{j=0}^{[\frac{\nu-n}{2}]}\tau_{j}e^{4j(j+n-\nu)t}P_{j}^{(n-1,-\nu)}(\cosh2d(z,w))\\
\nonumber&+(1-<z,w>)^{-\nu}e^{-t(\nu-n)^{2}}\frac{\Gamma(n)}{2\pi^{n+1}2^{2(\nu-n)}}\\
&\times \int_{0}^{+\infty}e^{-t\lambda^{2}}\mid C_{\nu}(\lambda)\mid^{-2}\prescript{}{2}{F}_1^{}(\frac{n-\nu-i\lambda}{2},\frac{n-\nu+i\lambda}{2},n;-\sinh^{2}(d(z,w)))d\lambda.
\end{align}
where $\tau_{j}=\frac{2(\nu-n-2j)\Gamma(\nu-j)}{\pi^{n}\Gamma(\nu-n-j+1)}$ and $C_{\nu}(\lambda)$ is the Harish-Chandra function defined in (\ref{E3.12}).
\end{proposition}
\begin{proof}
The solution $u(t,z)$ is given by the action of the semigroup $e^{t\Delta_{\nu}}$ on the initial data $\varphi(z)$ as follows
\begin{align}\label{E4.4}
u(t,z)=e^{t\Delta_{\nu}}[\varphi](z).
\end{align}
Note that for the operators $\Delta_{\nu}$ and $\tilde{\Delta}_{\nu}=-(\Delta_{\nu}+(n-\nu)^{2})$, we have the following semigroup relation
\begin{align}\label{E4.5}
e^{t\Delta_{\nu}}=e^{-t(\nu-n)^{2}}e^{-t\tilde{\Delta}_{\nu}}.
\end{align}
Then, the heat kernel $k(t,z,w)$ of $\Delta_{\nu}$ is given by
\begin{align}\label{E4.6}
K_{\nu}(t,z,w)=e^{-t(\nu-n)^{2}}\tilde{K}_{\nu}(t,z,w),
\end{align}
where $\tilde{K}_{\nu}(t,z,w)$ is the heat kernel of the operator $\tilde{\Delta}_{\nu}$.
By using (\ref{E3.58}) and (\ref{E3.59}), we write
\begin{align}\label{E4.7}
\nonumber \tilde{K}_{\nu}(t,z,w)&=\int_{\sigma(\tilde{\Delta}_{\nu})}e_{s}^{\nu}(s,w,z)e^{-ts}ds\\
\nonumber &=\sum_{j=0}^{[\frac{\nu-n}{2}]}c_{j}\frac{j!}{(n)_{j}}(1-<z,w>)^{-\nu}P_{j}^{(n-1,-\nu)}(\cosh2d(z,w))e^{-ts_{j}}\\
&+\frac{\Gamma(n)}{4\pi^{n+1}2^{2(\nu-n)}}(1-<z,w>)^{-\nu}\int_{0}^{+\infty}|C_{\nu}(\sqrt{s})|^{-2}
(\sqrt{s})^{-1}\phi_{\sqrt{s}}^{(n-1,\nu)}(d(z,w))e^{-ts}ds.
\end{align}
By replacing the Jacobi function $\phi_{\sqrt{s}}^{(n-1,\nu)}(.)$ by its expression given in terms of the hypergeometric  function defined in (\ref{E3.10}) with using the change of variable $s=\lambda^{2}$  ($\lambda>0$) in the continuous part
and replacing $s_{j}$ by its value $s_{j}=-(\nu-n-2j)^{2}$ in the discrete part, we obtain
\begin{align}\label{E4.8}
\nonumber\tilde{K}_{\nu}(t,z,w)&=(1-<z,w>)^{-\nu}\sum_{j=0}^{[\frac{\nu-n}{2}]}c_{j}\frac{j!}{(n)_{j}}e^{(\nu-n-2j)^{2}t}
P_{j}^{(n-1,-\nu)}(\cosh2d(z,w))\\
\nonumber &+(1-<z,w>)^{-\nu}\frac{\Gamma(n)}{2\pi^{n+1}2^{2(\nu-n)}}\\
&\times \int_{0}^{+\infty}e^{-t\lambda^{2}}\mid C_{\nu}(\lambda)\mid^{-2}\prescript{}{2}{F}_1^{}(\frac{n-\nu-i\lambda}{2},\frac{n-\nu+i\lambda}{2},n;-\sinh^{2}(d(z,w)))d\lambda.
\end{align}
Next, by using the above equation and (\ref{E4.6}), we get
\begin{align}\label{E4.9}
\nonumber K_{\nu}(t,z,w) &= (1-<z,w>)^{-\nu}\sum_{j=0}^{[\frac{\nu-n}{2}]}\tau_{j}e^{4j(j+n-\nu)t}P_{j}^{(n-1,-\nu)}(\cosh(2d(z,w))\\
\nonumber &+(1-<z,w>)^{-\nu}\frac{\Gamma(n)}{2\pi^{n+1}2^{2(\nu-n)}}\\
\nonumber\\
 &\times e^{-t(\nu-n)^{2}}\int_{0}^{+\infty}e^{-t\lambda^{2}}\mid C_{\nu}(\lambda)\mid^{-2}\prescript{}{2}{F}_1^{}\left( \frac{n-\nu-i\lambda}{2},\frac{n-\nu+i\lambda}{2},n;-\sinh^{2}(d(z,w))\right) d\lambda,
\end{align}
where
\begin{align}\label{E4.10}
\tau_{j}=\frac{2(\nu-n-2j)\Gamma(\nu-j)}{\pi^{n}\Gamma(\nu-n-j+1)}.
\end{align}
\end{proof}
Since the operator ${\Delta_{\nu}}$ is self adjoint. Then its resolvent  operator
\begin{align}\label{E4.11}
R(\xi,{\Delta_{\nu}})=(\xi- {\Delta_{\nu}})^{-1}
\end{align}
is related to the semigroup operator by the following formula
\begin{align}\label{E4.12}
R(\xi,{\Delta_{\nu}})=\int_{0}^{+\infty}e^{-t\xi}T(t)dt,
\end{align}
where $T(t)$ is the heat semigroup $e^{t{\Delta_{\nu}}}$. As a direct consequence of the above proposition, we can derive the resolvent kernel of the $G-$invariant operator $\Delta_{\nu}$. Precisely, we have the following proposition.
\begin{proposition}\label{P4.2} Let $\xi\in\mathbb{C}$ such that
\begin{align}\label{E4.13}
Re(\xi)> \omega(\nu,n)=p_{n}^{\nu}-(n-\nu)^{2},
\end{align}
with
\begin{align}\label{E4.14}
p_{n}^{\nu}=\max\{\mid s_{j}\mid, \hspace{0.2cm}0\leq j <\frac{\nu-n}{2}\},
\end{align}
\begin{align}\label{E4.15}
s_{j}=-(2j+n-\nu)^{2}.
\end{align}
Then the resolvent operator $R(\xi,{\Delta_{\nu}})$ is given by
\begin{align}\label{E4.16}
 R(\xi,{\Delta_{\nu}})[\varphi](z)=\int_{\mathbb{B}_{n}}R_{\nu}(\xi;z,w)\varphi(w)d\mu_{\nu}(w),
 \end{align}
 where the resolvent kernel $R_{\nu}(\xi,z,w)$ is given as follows
\begin{align}\label{E4.17}
\nonumber R(\xi,z,w)&=(1-\left<z,w\right>)^{-\nu}\sum_{j=0}^{[\frac{\nu-n}{2}]}\tau_j\frac{1}{\lambda_{j}^{2}+(\nu-n)^2+\xi}P_j^{(n-1,-\nu)}(\cosh(2d(z,w)))\\
\nonumber&+(1-\left<z,w\right>)^{-\nu}\frac{\Gamma(n)}{2\pi^{n+1}2^{2(\nu-n)}}\int_0^{+\infty}
\frac{|c_\nu(\lambda)|^{-2}}{\lambda^2+(\nu-n)^2+\xi}\\
& \times\prescript{}{2}{F}_1^{}\left(\frac{n-\nu-i\lambda}{2},\frac{n-\nu-i\lambda}{2},n,-\sinh^2(d(z,w))\right)d\lambda,
\end{align}
where  $\lambda_{j}=-i(\nu-n-2j)$ for $0\leq j<\frac{\nu-n}{2}$, $\tau_{j}=\frac{2(\nu-n-2j)\Gamma(\nu-j)}{\pi^{n}\Gamma(\nu-n-j+1)}$ and $C_{\nu}(\lambda)$ is the Harish-Chandra function defined in (\ref{E3.12}).
 \end{proposition}
  \begin{proof}
  The equation (\ref{E4.5}) can be rewritten as
\begin{align}\label{E4.18}
e^{t\Delta_{\nu}}=e^{t\omega(\nu,n)}e^{-t(\tilde{\Delta}_{\nu}+p^{\nu}_{n})}.
\end{align}
where $p^{\nu}_{n}=\max\{ |S_{j}|, \hspace{0.2cm}0\leq j<\frac{\nu-n}{2}\}$ and $\omega(\nu,n)= p^{\nu}_{n}-(\nu-n)^{2}$.
 We can easily see that the spectrum $\sigma(\tilde{\Delta}_{\nu}+p^{\nu}_{n})$  is included in the set of positive real numbers. Then by the self-adjointness, we show that $e^{-t(\tilde{\Delta_{\nu}}+p^{\nu}_{n})}$ is a contraction semigroup \cite[p.132]{Kon} for the general theory. Then, from equation (\ref{E4.18}), we get the following estimate
 \begin{align}\label{E4.19}
 \| e^{t\Delta_{\nu}}\| \leq e^{t\omega(\nu,n)}.
 \end{align}
 This last estimate lead us to consider the the following integral
\begin{align}\label{E4.20}
\int_{0}^{+\infty}e^{-t\xi}K_{\nu}(t,z,w)dt,
\end{align}
which takes a sense for $Re(\xi)>\omega(\nu,n)$ and can reduce to the following expression
\begin{align}\label{E4.21}
\nonumber \int_{0}^{+\infty}e^{-t\xi}K_{\nu}(t,z,w)dt&= (1-\left<z,w\right>)^{-\nu}\sum_{j=0}^{[\frac{\nu-n}{2}]}\tau_j\frac{1}{\xi+4j(\nu-n-j)}P_j^{(n-1,-\nu)}(\cosh(2d(z,w)))\\
\nonumber&+(1-\left<z,w\right>)^{-\nu}\frac{\Gamma(n)}{2\pi^{n+1}2^{2(\nu-n)}}\int_0^{+\infty}\frac{|c_\nu(\lambda)|^{-2}}
{\xi+(\nu-n)^2+\lambda^2}\\
 &\times\prescript{}{2}{F}_1^{}\left(\frac{n-\nu-i\lambda}{2},\frac{n-\nu-i\lambda}{2},n,-\sinh^2(d(z,w))\right)d\lambda.
  \end{align}
Then, by \cite[p.55]{Nag}, we have
\begin{align}\label{E4.22}
\{\xi\in\mathbb{C},\hspace{0.2cm} Re(\xi)> \omega(\nu,n)\}\subset\rho(\Delta_{\nu}),
\end{align}
\begin{align}\label{E4.23}
 \nonumber R_{\nu}(\xi,z,w)&= (1-\left<z,w\right>)^{-\nu}\sum_{j=0}^{[\frac{\nu-n}{2}]}\tau_j\frac{1}{\xi+4j(\nu-n-j)}P_j^{(n-1,-\nu)}(\cosh(2d(z,w)))\\
\nonumber &+(1-\left<z,w\right>)^{-\nu}\frac{\Gamma(n)}{2\pi^{n+1}2^{2(\nu-n)}}\int_0^{+\infty}\frac{|c_\nu(\lambda)|^{-2}}{\xi+
(\nu-n)^2+\lambda^2}\\
 &\times\prescript{}{2}{F}_1^{}\left(\frac{n-\nu-i\lambda}{2},\frac{n-\nu-i\lambda}{2},n,-\sinh^2(d(z,w))\right)d\lambda.
\end{align}
Then, the above equation can be also rewritten also as follows
\begin{align}\label{E4.24}
\nonumber R(\xi,z,w)&=(1-\left<z,w\right>)^{-\nu}\sum_{j=0}^{[\frac{\nu-n}{2}]}\tau_j\frac{1}{\lambda_{j}^{2}+(\nu-n)^2+\xi}P_j^{(n-1,-\nu)}(\cosh(2d(z,w)))\\
\nonumber&+(1-\left<z,w\right>)^{-\nu}\frac{\Gamma(n)}{2\pi^{n+1}2^{2(\nu-n)}}\int_0^{+\infty}
\frac{|c_\nu(\lambda)|^{-2}}{\lambda^2+(\nu-n)^2+\xi}\\
& \times\prescript{}{2}{F}_1^{}\left(\frac{n-\nu-i\lambda}{2},\frac{n-\nu-i\lambda}{2},n,-\sinh^2(d(z,w))\right)d\lambda,
\end{align}
where  $\lambda_{j}=-i(\nu-n-2j)$ for $0\leq j<\frac{\nu-n}{2}$, $\tau_{j}=\frac{2(\nu-n-2j)\Gamma(\nu-j)}{\pi^{n}\Gamma(\nu-n-j+1)}$ and $C_{\nu}(\lambda)$ is the Harish-Chandra function defined in (\ref{E3.12}).
\end{proof}
\section{Wave kernel}
Now, we consider the following Cauchy problem of wave equation with the operator $\Delta_{\nu}$ on $\mathbb{B}_{n}$:
\begin{align}\label{E5.1}
\left\{
  \begin{array}{ll}
    \dfrac{\partial ^{2} u(t,z)}{\partial t^{2}}=\vartriangle _{\nu }(t,z),\hspace{0.2cm}(t,z)\in \mathbb{R}\times\mathbb{B}_{n},&\hbox{} \\\\
   u(0,z)=0 ,\hspace{0.2cm}\dfrac{\partial u(0,z) }{\partial t}=f(z)\in C^{\infty}_{0}(\mathbb{B}_{n}). & \hbox{}
  \end{array}
\right.
\end{align}
\begin{proposition}\label{P5.1} The solution of the Cauchy problem $(\ref{E5.1})$ is given explicitly by the following integral formula
\begin{align}\label{E5.2}
u(t,z)=\int_{\mathbb{B}_{n}}W_{\nu}(t,z,w)\varphi(w)d\mu_{\nu}(w),
\end{align}
where $ W_{\nu}(t,z,w)$ is the wave  kernel given as follows
\begin{align}\label{E5.3}
\nonumber W_{\nu}&(t,z,w)=(1-<z,w>)^{-\nu}\sum_{j=0}^{[\frac{\nu-n}{2}]}\tau_{j} \dfrac{\sin(2t\sqrt{j(\nu-n-j)}}{2\sqrt{j(\nu-n-j)}}P_{j}^{(n-1,-\nu)}(\cosh2d(z,w))\\
\nonumber &+(1-<z,w>)^{-\nu}\frac{\Gamma(n)}{2\pi^{n+1}2^{2(\nu-n)}}\\
\nonumber\\
&\times \int_{0}^{+\infty}\dfrac{\sin(t\sqrt{\lambda^2+(n-\nu)^{2}})}{\sqrt{\lambda^2+(n-\nu)^{2}}}\mid C_{\nu}(\lambda)\mid^{-2}\prescript{}{2}{F}_1^{}(\frac{n-\nu-i\lambda}{2},\frac{n-\nu+i\lambda}{2},n;-\sinh^{2}(d(z,w)))d\lambda.
\end{align}
where $\tau_{j}=\frac{2(\nu-n-2j)\Gamma(\nu-j)}{\pi^{n}\Gamma(\nu-n-j+1)}$ and $C_{\nu}(\lambda)$ is the Harish-Chandra function defined in (\ref{E3.12}).
\end{proposition}
\begin{proof} By using formula $(\ref{E3.1})$, the operator $\Delta_{\nu}$ can be also rewritten as
\begin{align}\label{E5.4}
\Delta_{\nu}=-(\tilde{\Delta}_{\nu}+(n-\nu)^{2}).
\end{align}
Then, the above Cauchy problem becomes
\begin{align}\label{E5.5}
\left\{
  \begin{array}{ll}
    \dfrac{\partial ^{2} u(t,z)}{\partial t^{2}}+(\tilde{\Delta}_{\nu}+(n-\nu)^{2})u(t,z),\hspace{0.2cm}(t,z)\in \mathbb{R}\times\mathbb{B}_{n},&\hbox{} \\\\
   u(0,z)=0 ,\hspace{0.2cm}\dfrac{\partial u(0,t)}{\partial t}=f(z)\in C^{\infty}_{0}(\mathbb{B}_{n}). & \hbox{}
  \end{array}
\right.
\end{align}
Now, thanks to formulas $(\ref{E3.58})$ and $(\ref{E3.59})$, the solution of the above wave Cauchy problem is given by
\begin{align}\label{E5.6}
u(t,z)=\int_{\mathbb{B}_{n}}W_{\nu}(t,z,w)\varphi(w)d\mu_{\nu}(w),
\end{align}
where  the wave distributional kernel is given by
\begin{align}\label{E5.7}
\nonumber &W_{\nu}(t,z,w)=\int_{\sigma(\tilde{\Delta}_{\nu})}e^{\nu}(s,z,w)\frac{\sin(\sqrt{s+(n-\nu)^{2}})}{\sqrt{s+(n-\nu)^{2}}}ds\\
\nonumber &:=\frac{\Gamma(n)}{4\pi^{n+1}2^{2(\nu-n)}}(1-<z,w>)^{-\nu}\int_{0}^{+\infty}|C_{\nu}(\sqrt{s})|^{-2}(\sqrt{s})^{-1}
\phi_{\sqrt{s}}^{(n-1,-\nu)}(d(z,w))\frac{\sin(\sqrt{s+(n-\nu)^{2}})}{\sqrt{s+(n-\nu)^{2}}}ds\\
&+\sum_{j=0}^{[\frac{\nu-n}{2}]}c_{j}\frac{j!}{(n)_{j}}(1-<z,w>)^{-\nu}
P_{j}^{(n-1,-\nu)}(\cosh(2d(z,w)))\frac{\sin(\sqrt{s_{j}+(n-\nu)^{2}})}{\sqrt{s_{j}+(n-\nu)^{2}}},
\end{align}
By using the fact that $s_{j}=\lambda_{j}^{2}=-(\nu-n-2j)^{2}$ for $j=0,1,...,[\frac{\nu-n}{2}]$ combined with the change of variable $s=\lambda^{2}$, $\lambda>0$, the above wave kernel $W_{\nu}(t,z,w)$ becomes
\begin{align}\label{E5.8}
\nonumber W_{\nu}&(t,z,w)=(1-<z,w>)^{-\nu}\sum_{j=0}^{[\frac{\nu-n}{2}]}c_{j}\frac{j!}{(n)_{j}} \dfrac{\sin(2t\sqrt{j(\nu-n-j)}}{2\sqrt{j(\nu-n-j)}}P_{j}^{(n-1,-\nu)}(\cosh2d(z,w))\\
\nonumber &+(1-<z,w>)^{-\nu}\frac{\Gamma(n)}{2\pi^{n+1}2^{2(\nu-n)}}\\
\nonumber\\
&\times \int_{0}^{+\infty}\dfrac{\sin(t\sqrt{\lambda^2+(n-\nu)^{2}})}{\sqrt{\lambda^2+(n-\nu)^{2}}}\mid C_{\nu}(\lambda)\mid^{-2}\prescript{}{2}{F}_1^{}(\frac{n-\nu-i\lambda}{2},\frac{n-\nu+i\lambda}{2},n;-\sinh^{2}(d(z,w)))d\lambda.
\end{align}
Taking into account that
$\tau_{j}:=c_{j}\frac{j!}{(n)_{j}}=\frac{2(\nu-n-2j)\Gamma(\nu-j)}{\pi^{n}\Gamma(\nu-n-j+1)}$, we get the desired result.
\end{proof}
                            \section{Explicit formula for some generalized integrals}
In this section, we will give some generalized integrals. The fist formula will be established by comparing the integral wave kernel obtained by the use of the spectral density with the explicit wave kernel given in \cite{Int}. The second formula will be given by comparing the integral resolvent kernel obtained from the heat kernel with the Green kernel ("resolvent kernel") obtained in \cite{Gha}. Theses two formulas can be considered as an addition to the table of integral formulas for the special functions.
\begin{proposition}\label{P6.1} Let $\nu>n$, $\nu\in \mathbb{R}\setminus\mathbb{Z}$ and $0\leq x<\sinh(|t|)$ for $t\in \mathbb{R}$, then we have the following integral formula
\begin{align}\label{E6.1}
\nonumber &\int_{0}^{+\infty}\left|\frac{\Gamma(\frac{i\lambda+n-\nu}{2})\Gamma(\frac{i\lambda+n+\nu}{2})}{\sqrt{\lambda}\Gamma(i\lambda)}\right|^{2}
\times\prescript{}{2}{F}_1^{}\left( \frac{n-\nu-i\lambda}{2},\frac{n-\nu+i\lambda}{2},n;-x\right)\sin(t\lambda)d\lambda\\
\nonumber &=(-1)^{n-1}\pi\frac{\Gamma(n-\frac{1}{2})}{\Gamma(n)}(1+x)^{\frac{\nu-n}{2}}(\frac{\sinh^{2}(t)}{1+x}-1)^{-n+\frac{1}{2}}_{+}\\
\nonumber &\times\prescript{}{2}{F}_1^{}\left(1-n+\nu,1-n-\nu,\frac{3}{2}-n,\frac{1}{2}-\frac{\cosh(t)}{2\sqrt{1+x}}\right)\\
&-2^{2(\nu-n+1)}\frac{\pi}{\Gamma(n)}\sum_{j=0}^{[\frac{\nu-n}{2}]}\frac{(\nu-n-2j)\Gamma(\nu-j)}{\Gamma(\nu-n-j+1)} P^{(n-1,-\nu)}_{j}(2x+1)\frac{\sinh(t(2j+n-\nu))}{2j+n-\nu},
\end{align}
where the notation $(x)_{+}$ means the positive real part of the real number $x$.
\end{proposition}
\begin{proof} Recall that the following wave Cauchy problem  $(W_{n}^{\alpha\beta})$
\begin{align*}
(W^{\alpha\beta}_{n}):\left\{
  \begin{array}{ll}
    \dfrac{\partial ^{2} u(t,z)}{\partial t^{2}}=\vartriangle_{\alpha\beta}u(t,z), \hspace{0.2cm}(t,z)\in \mathbb{R}\times\mathbb{B}_{n},&  \hbox{} \\\\
   u(0,z)=0 ,\hspace{0.2cm}\dfrac{\partial u(0,z) }{\partial t}=f(z)\in C^{\infty}_{0}(\mathbb{B}_{n}), & \hbox{}
  \end{array}
\right.
\end{align*}
have been considered in \cite{Int}, where the partial differential operator
\begin{align}\label{E6.2}
\Delta_{\alpha\beta} =4(1-|z|^{2})\{\sum_{1\leq i,j\leq n}(\delta_{i,j}-z_{i}\overline{z}_{j})\frac{\partial^{2}}{\partial z_{i}\partial\overline{z}_{j}}+\alpha \sum_{j=1}^{n}z_{j}\frac{\partial}{\partial z_{j}}+\beta \sum_{j=1}^{n}\overline{z}_{j}\frac{\partial}{\partial\overline{z}_{j}}-\alpha\beta\}+\sigma_{\alpha\beta}^{2}
\end{align}
acts on the Hilbert space $L^{2}(\mathbb{B}_{n},d\mu_{\alpha\beta}(z))$, with $d\mu_{\alpha\beta}(z)=(1-\mid z\mid^{2})^{-(\alpha+\beta+n)-1}dm(z)$, $\alpha$, $\beta\in\mathbb{R}$, $dm(z)$ is the Lebesgue measure on $\mathbb{C}^{n}$ and $\sigma_{\alpha\beta}^{2}=(\alpha+\beta+n)^{2}$. The authors have proved that the solution of the Cauchy problem
 $(W^{\alpha\beta}_{n})$ is given by
\begin{align}\label{E6.3}
\nonumber u(t,z)&=(2\pi)^{-n}(\frac{\partial}{\sinh(t)\partial t})^{n-1}\int_{\mathbb{B}_{n}}K_{1}^{\alpha\beta}(t,z,w)f(w)d\mu_{\alpha\beta}(z)\\
&=\int_{\mathbb{B}_{n}}K_{n}^{\alpha\beta}(t,z,w)f(w)d\mu_{\alpha\beta}(z).
\end{align}
where $K_{1}^{\alpha\beta}(t,z,w)$ is the kernel given in the equation $(2.2)$ of the reference \cite{Int} and the distributional kernel $K_{n}^{\alpha\beta}(t,z,w)$ is given by
\begin{align}\label{E6.4}
K_{n}^{\alpha\beta}(t,z,w)=c_{n}(1-\overline{<z,w>})^{\alpha}(1-<z,w>)^{\beta}I_{n}^{\alpha\beta}(t,\rho(z,w)),
\hspace{0.2cm}\rho(z,w)<|t|,
\end{align}
with $\rho(z,w)$ means the Bargmann distance $d(z,w)$  defined in (\ref{E1.16}) and constant $c_{n}$ is given by
\begin{align}\label{E6.5}
c_{n}=(-1)^{n-1}\frac{1}{2\pi^{n}}\Gamma(n-\frac{1}{2}).
\end{align}
The function $I_{n}^{\alpha\beta}(\rho,t)$ is defined by
\begin{align}\label{E6.6}
\nonumber I_{n}^{\alpha\beta}(\rho,t)=(\cosh(\rho))^{-(n+\alpha+\beta)}(\frac{\cosh^{2}(t)}{\cosh^{2}(d(z,w))}-1)_{+}^{-n+\frac{1}{2}}\\
\times\prescript{}{2}{F}_1^{}(a,b;\frac{a+b+1}{2},\frac{\cosh(d(z,w))-\cosh(t)}{2\cosh(d(z,w))}),
\end{align}
with $a=1-n+\alpha-\beta$ and $b=1-n+\beta-\alpha$. Note that for $\alpha=0$ and $\beta=-\nu$, the wave kernel for the Cauchy problem $W_{n}^{0,-\nu}$ is given by
\begin{align}\label{E6.7}
\nonumber K_{n}^{0,-\nu}(t,z,w)&=\frac{(-1)^{n-1}\Gamma(n-\frac{1}{2})}{2\pi^{n}}(1-<z,w>)^{-\nu}
(\cosh(d(z,w)))^{\nu-n}(\frac{\cosh^{2}(t)}{\cosh^{2}(d(z,w))}-1)_{+}^{-n+\frac{1}{2}}\\
&\times\prescript{}{2}{F}_1^{}(1-n+\nu,1-n-\nu;\frac{3}{2}-n,\frac{\cosh(d(z,w))-\cosh(t)}{2\cosh(d(z,w))}).
\end{align}
Furthermore, the Cauchy problem $W^{0,-\nu}_{n}$ can be also rewritten as
\begin{align}\label{E6.8}
\left\{
  \begin{array}{ll}
    \dfrac{\partial ^{2} u(t,z)}{\partial t^{2}}+\tilde{\Delta}_{\nu}u(t,z)=0, \hspace{0.2cm}(t,z)\in  \mathbb{R}\times\mathbb{B}_{n},&  \hbox{} \\\\
   u(0,z)=0 ,\hspace{0.2cm}\dfrac{\partial u(0,z)}{\partial t}=f(z)\in C^{\infty}_{0}(\mathbb{B}_{n}), & \hbox{}
  \end{array}
\right.
\end{align}
where $\tilde{\Delta}_{\nu}$ is the operator defined in $(\ref{E1.17})$. Then, by using formulas  $(\ref{E3.47})$ and  $(\ref{E3.48})$, the solution of the above wave Cauchy problem is given by
\begin{align}\label{E6.9}
\nonumber u(t,z) &= \dfrac{\sin(t\sqrt{\tilde{\Delta}_{\nu})}}{\sqrt{\tilde{\Delta}_{\nu}}}[f](z) \\
&=\int_{\mathbb{\Bbb B}^{n}}W^{0,-\nu}_{n}(t,z,w)f(w)d\mu_{\nu}(w),
\end{align}
where the distributional wave kernel $W^{0,-\nu}_{n}(t,z,w)$ is given by
\begin{align}\label{E6.10}
\nonumber W^{0,-\nu}_{n}(t,z,w)&=\int_{\sigma(\tilde{\Delta}_{\nu})}e^{\nu}(s,w,z)\frac{\sin(t\sqrt{s})}{\sqrt{s}}ds\\
\nonumber&:=(1-<z,w>)^{-\nu}\sum_{j=0}^{[\frac{\nu-n}{2}]}c_{j}\frac{j!}{(n)_j} P^{(n-1,-\nu)}_j(\cosh(2d(z,w)))\frac{sin(t\sqrt{s_{j}})}{\sqrt{s_{j}}}\\
&+\frac{\Gamma(n)}{4\pi^{n+1}2^{2(\nu-n)}}(1-<z,w>)^{-\nu}\int_{0}^{+\infty}|C_{\nu}(\sqrt{s})|^{-2}(\sqrt{s})^{-1}
\phi_{\sqrt{s}}^{(n-1,-\nu)}(d(z,w))\dfrac{\sin(t\sqrt{s})}{\sqrt{s}}ds.
\end{align}
Then, by using the change of variable $s=\lambda^{2}$, $(\lambda\geq0)$, and the fact that\\ $\sqrt{s_{j}}=\lambda_{j}=i(2j+n-\nu)$ for
$j=0,1,...,[\frac{\nu-n}{2}]$ combined the well known formula
 \begin{align}\label{E6.11}
\sin(ix)=i\sinh(x),\hspace{0.2cm}x\in \mathbb{R}.
\end{align}
The equation  $(\ref{E6.10})$ becomes
\begin{align}\label{E6.12}
\nonumber W^{0,-\nu}_{n}(t,z,w)&=(1-<z,w>)^{-\nu}\sum_{j=0}^{[\frac{\nu-n}{2}]}c_{j}\frac{j!}{(n)_j} P^{(n-1,-\nu)}_j(\cosh((2d(z,w))))\frac{\sinh(t(2j+n-\nu))}{2j+n-\nu}\\
\nonumber &+\frac{\Gamma(n)}{2\pi^{n+1}2^{2(\nu-n)}}(1-<z,w>)^{-\nu}\\
&\times \int_{0}^{+\infty}\mid C_{\nu}(\lambda)\mid^{-2}{\dfrac{\sin(t\lambda)}{\lambda}}\prescript{}{2}{F}_1^{}\left( \frac{n-\nu-i\lambda}{2},\frac{n-\nu+i\lambda}{2},n;-sinh^{2}(d(z,w))\right) d\lambda.
\end{align}
Now, by using the uniqueness of the wave kernel for the Cauchy problem $W^{0,-\nu}_{n}$, we get the following equality
\begin{align}\label{E6.13}
W^{0,-\nu}_{n}(t,z,w)=K^{0,-\nu}_{n}(t,z,w).
\end{align}
This last equation lead us to write
\begin{align}\label{E6.14}
\nonumber&\frac{\Gamma(n)}{2\pi^{n+1}2^{2(\nu-n)}}\int_{0}^{+\infty}\mid C_{\nu}(\lambda)\mid^{-2}{\dfrac{\sin(t\lambda)}{\lambda}}\prescript{}{2}{F}_1^{}\left( \frac{n-\nu-i\lambda}{2},\frac{n-\nu+i\lambda}{2},n;-sinh^{2}(d(z,w))\right) d\lambda\\
\nonumber&=\frac{(-1)^{n-1}\Gamma(n-\frac{1}{2})}{2\pi^{n}}
(\cosh(d(z,w)))^{\nu-n}(\frac{\cosh^{2}(t)}{\cosh^{2}(d(z,w))}-1)_{+}^{-n+\frac{1}{2}}\\
\nonumber&\times\prescript{}{2}{F}_1^{}(1-n+\nu,1-n-\nu;\frac{3}{2}-n,\frac{\cosh(d(z,w))-\cosh(t)}{2\cosh(d(z,w))})\\
&-\sum_{j=0}^{[\frac{\nu-n}{2}]}c_{j}\frac{j!}{(n)_j} P^{(n-1,-\nu)}_j(\cosh((2d(z,w))))\frac{\sinh(t(2j+n-\nu))}{2j+n-\nu}.
\end{align}
Now, setting $x=\sinh^{2}(d(z,w))$ then $\cosh((2d(z,w)))=2x+1$. The last condition in equation $(\ref{E6.4})$ require that $0\leq x<\sinh(|t|)$. So, by replacing the Harish-Chandra function $C_{\nu}(\lambda)$ by its expression given in $(\ref{E3.12})$, we get the desired formula.
\end{proof}
\begin{proposition}\label{P6.2} Let $\nu>n$, $\nu\in \mathbb{R}\setminus\mathbb{Z}$ and $\mu$ is a complex number such that $\mu\neq-i(2\ell+n\pm\nu)$ for $\ell=0,1,2,...$ and $Re(\mu^{2})<-P_{n}^{\nu}$, with $p_{n}^{\nu}=\max\{\mid s_{j}\mid,\hspace{0.2cm}0\leq j <\frac{\nu-n}{2}\}$ and $s_{j}=-(2j+n-\nu)^{2}$. Then, we have the following integral formula
\begin{align}\label{E6.15}
\nonumber&\int_0^{+\infty}\left|\frac{\Gamma(\frac{i\lambda+n-\nu}{2})
\Gamma(\frac{i\lambda+n+\nu}{2})}{\Gamma(i\lambda)}\right|^{2}\frac{1}{\lambda^2-\mu^2}\times
 \prescript{}{2}{F}_1^{}\left(\frac{n-\nu-i\lambda}{2},\frac{n-\nu+i\lambda}{2},n,-x\right)d\lambda\\
\nonumber&=\pi2^{2(\nu-n)}\frac{\Gamma(\frac{n-i\mu+\nu}{2})\Gamma(\frac{n-i\mu-\nu}{2})}{\Gamma(n)\Gamma(1-i\mu)}
(1+x)^{\frac{\nu+i\mu}{4}-\frac{n}{2}}\\
\nonumber&\times\prescript{}{2}{F}_1^{}\left(\frac{n-i\mu+\nu}{2},\frac{n-i\mu-\nu}{2},1-i\mu,\frac{1}{1+x}\right)\\
&-\frac{4\pi}{\Gamma(n)}2^{2(\nu-n)}\sum_{j=0}^{[\frac{\nu-n}{2}]}\frac{(\nu-n-2j)\Gamma(\nu-j)}{\Gamma(\nu-n-j+1)}\frac{1}
{\lambda_{j}^{2}-\mu^{2}}P_j^{(n-1,-\nu)}(2x+1).
\end{align}
\end{proposition}
\begin{proof}
First, let us consider the following partial differential operator
\begin{align}\label{E6.16}
H_{\frac{\nu}{2},1}=-4(1-\mid z\mid^{2})[\sum_{1\leq i,j\leq n}(\delta_{ij}-z_{i}\bar{z_{j}})
\frac{\partial^{2}}{\partial z_{i}\partial\bar{z_{j}}}+\frac{\nu}{2}\sum_{j=1}^{n}(z_{j}\frac{\partial}{\partial z_{j}}-
\bar{z_{j}}\frac{\partial}{\partial\bar{z_{j}}})]+\nu^{2}\mid z\mid^{2}
\end{align}
acting on the Hilbert space $L^{2}(\mathbb{B}_{n},(1-\mid z\mid^{2})^{-n-1}dm(z))$. The above operator represents the operator
$H_{B,\rho}$ defined in \cite{Gha}, for $B=\frac{\nu}{2}$ and $\rho=1$. The authors in \cite{Gha} have computed the Green kernel $R_{\frac{\nu}{2},1}(\mu,z,w)$, which have called abusively the resolvent kernel. They have introduce the kernel $R_{\frac{\nu}{2},1}(\mu,z,w)$ as the right inverse \cite{Bea2} of the operators
\begin{align}\label{E6.17}
[H_{\frac{\nu}{2},1}-h_{\frac{\nu}{2},1}(\mu)],\hspace{0.2cm}h_{\frac{\nu}{2},1}(\mu)=\mu^{2}+\nu^{2}+n^{2}.
\end{align}
Precisely, that is the integral operator
\begin{align}\label{E6.18}
R_{\frac{\nu}{2},1}[F](z):=\int_{\mathbb{B}_{n}}R_{\frac{\nu}{2},1}(\mu,z,w)F(w)(1-\mid w\mid^{2})^{-n-1}dm(w)
\end{align}
acting on $L^{2}(\mathbb{B}_{n},(1-\mid z\mid^{2})^{-n-1}dm(z))$, which solves the partial differential equations
\begin{align}\label{E6.19}
[H_{\frac{\nu}{2},1}-h_{\frac{\nu}{2},1}(\mu)]g=F,\hspace{0.2cm}\mbox{for}\hspace{0.2cm}F\in L^{2}(\mathbb{B}_{n},(1-\mid z\mid^{2})^{-n-1}dm(z)).
\end{align}
Based on \cite{Gha}, the explicit expression
of the kernel $R_{\frac{\nu}{2},1}(\mu,z,w)$ is given by the following formula
\begin{align}\label{E6.20}
\nonumber R_{\frac{\nu}{2},1}(\mu,z,w)&=C_{\frac{\nu}{2},1}^{n}(\mu)\left(\frac{1-\overline{<z,w>}}{1-<z,w>}\right)^{\frac{\nu}{2}}
\left(\frac{(1-\mid z\mid^{2})(1-\mid w\mid^{2})}{\mid1-<z,w>\mid^{2}}\right)^{n-i\frac{\mu}{2}}\\
&\times\prescript{}{2}{F}_1^{}\left(\frac{n-i\mu+\nu}{2},\frac{n-i\mu-\nu}{2},1-i\mu,\frac{(1-\mid z\mid^{2})(1-\mid w\mid^{2})}{\mid1-<z,w>\mid^{2}}\right),
\end{align}
with the constant $C_{\frac{\nu}{2},1}^{n}(\mu)$ is given by
\begin{align}\label{E6.21}
C_{\frac{\nu}{2},1}^{n}(\mu)=\frac{1}{2\pi^{n}}\frac{\Gamma(\frac{n-i\mu+\nu}{2})
\Gamma(\frac{n-i\mu-\nu}{2})}{\Gamma(1-i\mu)},
\end{align}
and the parameter $\mu$ must satisfy the following condition
\begin{align}\label{E6.22}
\mu\neq-i(2\ell+n\pm\nu), \hspace{0.2cm}\ell=0,1,2,...\textbf{.}
\end{align}
Now, by using Lemma $(\ref{L2.1})$ for $\alpha=0$, $\beta=-\nu$ and $\gamma=\frac{-\nu}{2}$, we obtain the following
formula
\begin{align}\label{E6.23}
\Delta_{\nu}=(1-\mid z\mid^{2})^{\frac{-\nu}{2}}\left(\Delta_{\frac{\nu}{2},\frac{-\nu}{2}}+2\nu(n-\frac{\nu}{2})I\right)(1-\mid z\mid^{2})^{\frac{\nu}{2}}.
\end{align}
Note that the operator $H_{\frac{\nu}{2},1}$ can be also rewritten as
\begin{align}\label{E6.24}
H_{\frac{\nu}{2},1}=-\Delta_{\frac{\nu}{2},\frac{-\nu}{2}}+\nu^{2}I,
\end{align}
then the equation $(\ref{E6.23})$ becomes
\begin{align}\label{E6.25}
\Delta_{\nu}=(1-\mid z\mid^{2})^{\frac{-\nu}{2}}\left(-H_{\frac{\nu}{2},1}+2n\nu I\right)(1-\mid z\mid^{2})^{\frac{\nu}{2}}
\end{align}
which is equivalent to
\begin{align}\label{E6.26}
(1-\mid z\mid^{2})^{\frac{\nu}{2}}\left(-\Delta_{\nu}+2n\nu I\right)(1-\mid z\mid^{2})^{\frac{-\nu}{2}}=H_{\frac{\nu}{2},1}.
\end{align}
Using this above formula and returning back to equation $(\ref{E6.19})$, so this last equation becomes equivalent to the following equation
\begin{align}\label{E6.27}
\left[(2n\nu-h_{\frac{\nu}{2},1}(\mu))I-\Delta_{\nu}  \right](1-\mid z\mid^{2})^{\frac{-\nu}{2}}g=(1-\mid z\mid^{2})^{\frac{-\nu}{2}}F.
\end{align}
Now, by setting $\xi=\xi(\mu):=2n\nu-h_{\frac{\nu}{2},1}(\mu)$ in Proposition $(\ref{P4.1})$ under the condition
\begin{align}\label{E6.28}
Re(\xi(\mu))>\omega(\nu,n),
\end{align}
where $\omega(\nu,n)$ is the constant defined in $(\ref{E4.13})$, the equation $(\ref{E6.27})$ becomes equivalent to
\begin{align}\label{E6.29}
g(z)=\int_{\mathbb{B}_{n}}(1-\mid z\mid^{2})^{\frac{\nu}{2}}R(\xi(\mu),z,w)(1-\mid w\mid^{2})^{\frac{-\nu}{2}}    F(w)(1-\mid w\mid^{2})^{\nu-n-1}dm(w),
\end{align}
where $R(\xi(\mu),z,w)$ is the resolvent kernel given by the following formula
\begin{align}\label{E6.30}
\nonumber R(\xi(\mu),z,w)&=(1-\left<z,w\right>)^{-\nu}\sum_{j=0}^{[\frac{\nu-n}{2}]}\frac{\tau_j}
{\lambda_{j}^{2}-\mu^{2}}P_j^{(n-1,-\nu)}
(\cosh(2d(z,w)))\\
\nonumber&+(1-\left<z,w\right>)^{-\nu}\frac{\Gamma(n)}{2\pi^{n+1}2^{2(\nu-n)}}\\
 &\times \int_0^{+\infty}\frac{|c_\nu(\lambda)|^{-2}}{\lambda^2-\mu^2}\times
 \prescript{}{2}{F}_1^{}\left(\frac{n-\nu-i\lambda}{2},\frac{n-\nu+i\lambda}{2},n,-\sinh^2(d(z,w))\right)d\lambda.
\end{align}
where $\tau_{j}=\frac{2(\nu-n-2j)\Gamma(\nu-j)}{\pi^{n}\Gamma(\nu-n-j+1)}$ and $\lambda_{j}^{2}=-(\nu-n-2j)^{2}$.\\
Now, returning back to equation $(\ref{E6.19})$ and requiring that the parameter $\mu$ satisfy the both conditions
\begin{align}\label{E6.31}
\mu\neq-i(2\ell+n\pm\nu), \hspace{0.2cm}\mbox{for} \hspace{0.2cm}\ell=0,1,2,...\hspace{0.2cm}\mbox{and} \hspace{0.2cm} Re(\xi(\mu))>\omega(\nu,n).
\end{align}
Then the equation $(\ref{E6.19})$ becomes equivalent to
\begin{align}\label{E6.32}
g(z)=R_{\frac{\nu}{2},1}[F](z)=\int_{\mathbb{B}_{n}}R_{\frac{\nu}{2}}(\mu,z,w)F(w)(1-\mid w\mid^{2})^{-n-1}dm(w).
\end{align}
Moreover, the right inverse of the operator $[H_{\frac{\nu}{2},1}-h_{\frac{\nu}{2},1}(\mu)]$, in the sense of \cite{Bea2},
becomes its inverse. Then by the consideration of equations $(\ref{E6.29})$ and $(\ref{E6.32})$, we get the following relation
\begin{align}\label{E6.33}
R_{\frac{\nu}{2},1}=(1-\mid z\mid^{2})^{\frac{\nu}{2}} R(\xi(\mu),z,w)(1-\mid w\mid^{2})^{\frac{\nu}{2}},
\end{align}
for $\mu\neq-i(2\ell+n\pm\nu)$ and  $Re(\xi(\mu))>\omega(\nu,n)$. Notice that this last condition is equivalent to
\begin{align}\label{E6.34}
Re(\mu^{2})<-P_{n}^{\nu},
\end{align}
where $P_{n}^{\nu}$ is the parameter defined in equation $(\ref{E4.14})$. Then, the equation $(\ref{E6.33})$ is true for the parameter $\mu$ in the domain
\begin{align}\label{E6.35}
D_{n,\nu}=\{\mu=u+iv,\hspace{0.2cm}u^{2}-v^{2}<-P_{n}^{\nu}\hspace{0.2cm}\mbox{and}\hspace{0.2cm}
\mu\neq-i(2\ell+n\pm\nu)\hspace{0.2cm}\mbox{for}\hspace{0.2cm}\ell=0,1,2,...\hspace{0.2cm}\}.
\end{align}
By returning back to  formula $(\ref{E6.33})$ and using equations $(\ref{E6.20})$ with $(\ref{E6.30})$, we obtain for $\mu\in D_{n,\nu}$ the following relation
\begin{align}\label{E6.36}
\nonumber&(1-\left<z,w\right>)^{-\nu}\sum_{j=0}^{[\frac{\nu-n}{2}]}\frac{\tau_j}
{\lambda_{j}^{2}-\mu^{2}}P_j^{(n-1,-\nu)}
(\cosh(2d(z,w)))\\
\nonumber&+(1-\left<z,w\right>)^{-\nu}\frac{\Gamma(n)}{2\pi^{n+1}2^{2(\nu-n)}}\\
 \nonumber&\times \int_0^{+\infty}\frac{|c_\nu(\lambda)|^{-2}}{\lambda^2-\mu^2}\times
 \prescript{}{2}{F}_1^{}\left(\frac{n-\nu-i\lambda}{2},\frac{n-\nu+i\lambda}{2},n,-\sinh^2(d(z,w))\right)d\lambda\\
\nonumber&=(1-\mid z\mid^{2})^{\frac{-\nu}{2}}(1-\mid w\mid^{2})^{\frac{-\nu}{2}}C_{\frac{\nu}{2},1}^{n}(\mu)\left(\frac{1-\overline{<z,w>}}{1-<z,w>}\right)^{\frac{\nu}{2}}
\left(\frac{(1-\mid z\mid^{2})(1-\mid w\mid^{2})}{\mid1-<z,w>\mid^{2}}\right)^{n-i\frac{\mu}{2}}\\
&\times\prescript{}{2}{F}_1^{}\left(\frac{n-i\mu+\nu}{2},\frac{n-i\mu-\nu}{2},1-i\mu,\frac{(1-\mid z\mid^{2})(1-\mid w\mid^{2})}{\mid1-<z,w>\mid^{2}}\right).
\end{align}
Is not hard to show that the right hand side of the above equation can be rewritten as
\begin{align}\label{E6.37}
\nonumber&(1-\left<z,w\right>)^{-\nu}C_{\frac{\nu}{2},1}^{n}(\mu)\left(\frac{(1-\mid z\mid^{2})(1-\mid w\mid^{2})}{\mid1-<z,w>\mid^{2}}\right)^{n-\frac{\nu}{2}-i\frac{\mu}{2}}\\
&\times\prescript{}{2}{F}_1^{}\left(\frac{n-i\mu+\nu}{2},\frac{n-i\mu-\nu}{2},1-i\mu,\frac{(1-\mid z\mid^{2})(1-\mid w\mid^{2})}{\mid1-<z,w>\mid^{2}}\right).
\end{align}
By setting $x=\sinh^2(d(z,w))$, then $\cosh(d(z,w))=(1+x)^{\frac{1}{2}}$ and $\cosh(2d(z,w))=2x+1$.
Finally, by replacing the right hand side of equation $(\ref{E6.36})$ by the last expression given in $(\ref{E6.37})$ and the Harish-Chandra $c_\nu(\lambda)$ by its expression  given in $(\ref{E3.12})$ combined with the fact that $\cosh^2(d(z,w))=   \frac{\mid1-<z,w>\mid^{2}}{(1-\mid z\mid^{2})(1-\mid w\mid^{2})}$, the equation  $(\ref{E6.36})$ becomes as follows
\begin{align}\label{E6.38}
\nonumber&\int_0^{+\infty}\left|\frac{\Gamma(\frac{i\lambda+n-\nu}{2})
\Gamma(\frac{i\lambda+n+\nu}{2})}{\Gamma(i\lambda)}\right|^{2}\frac{1}{\lambda^2-\mu^2}\times
 \prescript{}{2}{F}_1^{}\left(\frac{n-\nu-i\lambda}{2},\frac{n-\nu+i\lambda}{2},n,-x\right)d\lambda\\
\nonumber&=\pi2^{2(\nu-n)}\frac{\Gamma(\frac{n-i\mu+\nu}{2})\Gamma(\frac{n-i\mu-\nu}{2})}{\Gamma(n)\Gamma(1-i\mu)}
(1+x)^{\frac{\nu+i\mu}{4}-\frac{n}{2}}\\
\nonumber&\times\prescript{}{2}{F}_1^{}\left(\frac{n-i\mu+\nu}{2},\frac{n-i\mu-\nu}{2},1-i\mu,\frac{1}{1+x}\right)\\
&-\frac{4\pi}{\Gamma(n)}2^{2(\nu-n)}\sum_{j=0}^{[\frac{\nu-n}{2}]}\frac{(\nu-n-2j)\Gamma(\nu-j)}{\Gamma(\nu-n-j+1)}\frac{1}
{\lambda_{j}^{2}-\mu^{2}}P_j^{(n-1,-\nu)}(2x+1),
\end{align}
which is the desired formula.
\end{proof}

\end{document}